\documentclass[11pt, oneside]{amsart}
\usepackage{boxedminipage}
\usepackage{amsfonts}
\usepackage{amsmath}
\usepackage{amsrefs}
\usepackage{amssymb}
\usepackage{graphicx}
\usepackage{amsthm}
\usepackage{t1enc}
\usepackage{subfig}
\usepackage{tikz-cd}
\usepackage{tikz}
\usepackage{dsfont}
\usepackage{fancyhdr}
\usepackage[margin=1in]{geometry}

\usepackage{amstext, amsmath, amsthm, amscd, amssymb}
\usepackage{manfnt}
\usepackage{mathrsfs}
\usepackage{url}

\renewcommand{\setminus}{{\smallsetminus}}

\newcommand\blfootnote[1]{%
  \begingroup
  \renewcommand\thefootnote{}\footnote{#1}%
  \addtocounter{footnote}{-1}%
  \endgroup
}

\newcommand{\bp}{\begin{pmatrix}}
\newcommand{\ep}{\end{pmatrix}}
\newcommand{\be}{\begin{equation}}
\newcommand{\ee}{\end{equation}}

\newcommand{\ba}{\begin{array}}
\newcommand{\ea}{\end{array}}

\DeclareMathOperator{\PD}{PD}
\DeclareMathOperator{\Hom}{Hom}

\DeclareMathOperator\Ab{Ab}

\DeclareMathOperator\im{im}

\DeclareMathOperator\Id{Id}

\DeclareMathOperator\Aut{Aut}

\DeclareMathOperator\ev{ev}

\DeclareMathOperator\lk{lk}
\DeclareMathOperator\LM{LM}
\DeclareMathOperator\BL{Bl}
\DeclareMathOperator\SL{ALM}
\DeclareMathOperator\ESL{EASL}

\DeclareMathOperator\rit{right}
\DeclareMathOperator\lef{left}

\usepackage[utf8]{inputenc}

\usepackage{thmtools}
\declaretheorem[
style=plain,
name=Theorem,
numberwithin=section
]{theo}
\declaretheorem[
style=plain,
name=Proposition,
numberlike=theo
]{prop}
\declaretheorem[
style=plain,
name=Lemma,
numberlike=theo
]{lem}
\declaretheorem[
style=plain,
name=Corollary,
numberlike=theo
]{cor}
\declaretheorem[
style=definition,
name=Definition,
numberlike=theo
]{defn}
\declaretheorem[
style=definition,
name=Example,
numberlike=theo
]{exam}
\declaretheorem[
style=definition,
name=Remark,
numberlike=theo
]{rem}

\declaretheorem[
style=definition,
name=Construction,
numberlike=theo
]{con}
\declaretheorem[
style=definition,
name=Claim,
numberlike=theo
]{cla}

\numberwithin{equation}{section}

\title{Three-component link homotopy}
\author{scott Stirling}

\begin{document}
\begin{abstract}
    In 2019, Schneidermann and Teicher showed that the Kirk invariant classifies two-component link maps of two-spheres in the four-sphere up to link homotopy. In this paper, we
construct a three-component link homotopy invariant.  We construct two link maps
where each component has the same image, and apply our invariant to prove that nevertheless they are not link homotopic.
We develop tools to help distinguish between three-component link maps. We then construct a similar
invariant for three-component annular link maps. Towards the end of the paper we discuss how
 to generalise to an $n$-component link map invariant.
\end{abstract}

\maketitle

\section{Introduction}
A continuous map \blfootnote{2020 \emph{Mathematics Subject Classification}. 57N35, 57K45,  57N35}\blfootnote{e-mail: scottspeirsstirling@gmail.com}

\[
f=f_1\sqcup\cdots\sqcup f_n:\coprod_{i=1}^nS^{2}\rightarrow S^4,
\]
 which keeps disjoint components in the domain disjoint in the image, i.e $f_i(S^2)\cap f_j(S^2)=\phi$, is called a \emph{link map}. We consider link maps up to \emph{link homotopy}, i.e. a homotopy through link maps. 

The classical study of link homotopy started with Milnor in his Bachelor's thesis \cite{LGM}. He made use of a quotient of the fundamental group of the complement of three disjoint circles in $S^3$, which we now call the Milnor group. Using this group Milnor  classified three-component links in $S^3$ up to link homotopy. He did this by constructing a triple linking number which measures linking behaviour for links of three or more components. This number is only well defined modulo the greatest common divisor of the linking numbers between each pair of components. The triple linking number, along with the linking number of each two component sublink, classifies three-component links up to link homotopy. 

Since then, higher dimensions have been studied and the study of link maps of $S^2$ inside $S^4$ began with Fenn and Rolfsen who showed there exists a link homotopically non-trivial link map $f:S^2\coprod S^2~\rightarrow~S^4$ which contains a self-intersection on each sphere \cite{FR}. Self-intersections are a necessary condition for generic smooth link maps of $S^2$ in $S^4$ to be non-trivial \cite{BT}.

In 1988, Paul Kirk defined a two-component link map invariant
\[
\sigma:\LM_{2,2}^4\rightarrow \mathbb{Z}\left[\mathbb{Z}\right]\oplus\mathbb{Z}\left[\mathbb{Z}\right].
\]

 Schneiderman and Teichner showed that the Kirk invariant classifies two-component link maps up to link homotopy \cite{ST}, computing the group $\LM_{2,2}^4$. Hence, the Kirk invariant is analogous to the linking number in the classical two-component case. 
 
 To generalise the Kirk invariant we will use $\mathbb{Z}[F/F_3]$ in place of $\mathbb{Z}[\mathbb{Z}]$. The group $F/F_3$ is the third lower central series quotient of the free group on two generators - $F_3$ is the third lower central series subgroup. This is also the free Milnor group on two generators as proven in Lemma \ref{MF3}. The group $F/F_3$ admits the presentation
\[
F/F_3\cong\left<y,z, s\mid \left[z,s\right],\left[y,s\right], \left[y,z\right]s^{-1}\right>\text{,}
\]
as shown by subsection \ref{ff3}.
It can also be described as the following central extension
\[
\begin{tikzcd}
0\arrow[r]&\mathbb{Z}\arrow[r]&F/F_3\arrow[r]&\mathbb{Z}^2\arrow[r]&0\text{,}
\end{tikzcd}
\]
where the inclusion maps $1$ to $[y,z]$.
 In this article, we construct a function 
 \[
 \widetilde{\sigma}^3:\LM_{2,2,2}^4\rightarrow \widetilde{K}\text{,}
\]
where $\widetilde{K}:=\left(\mathbb{Z}[F/F_3]\right)^3/\sim$ for some choice of equivalence relation, specified in Section \ref{chap5}. Unlike the Kirk invariant, $\widetilde{\sigma}^3$ is not a group homomorphism as  $\LM_{2,2,2}^4$ cannot be turned into a group using connect sum as the operation is not well defined up to link homotopy, see Proposition~\ref{introconnect}.
This invariant is constructed by first considering an invariant of based three-component link maps up to link homotopy, which takes values in $\left(\mathbb{Z}\left[F/F_3\right]\right)^3$. These values are given by taking each component and computing the self-intersection number in the complement of the other two components. We then consider the effect of changing the basings of a link map on the invariant and determine the finest equivalence relation required to define an unbased invariant. The effect of changing the basing path transforms the group elements by multiplying group elements by
 the power of a commutator term. 

 In $\left(\mathbb{Z}\left[F/F_3\right]\right)^3$, the group our based link map invariant takes values in, the group generators we choose in each factor are different. In the first component we use the generators $y, z$ and $s=\left[y,z\right]$; in the second factor we have generators $z, x$ and $t=\left[z,x\right]$; and in the third factor we use the generators $x,y$ and $u=\left[x,y\right]$. In each of these, $x, y$ and $z$ represent meridians of the first, second, and third sphere respectively. 
Using $\widetilde{\sigma}^3$ we prove the following.
\begin{theo}\label{introthm}
There exists a three-component link map $f$ with a choice of basing path such that
\[
\widetilde{\sigma}^3(f)=\left(z\left(s-1\right)+z^{-1}\left(s^{-1}-1\right),\, 0,\, x\left(1-u\right)+x^{-1}\left(1-u^{-1}\right)\right).
\]
Furthermore, removal of any component gives a trivial two-component link map. But $f$ is not link homotopically trivial.
\end{theo}
This shows that the invariant can detect linking information which only occurs in three or more components. Using Theorem \ref{introthm} we prove the following 
\begin{theo}
For each $n\geq3$ there exists link maps $f=f_1\sqcup\cdots\sqcup f_n$ and $f^\prime=f_1^\prime\sqcup\cdots\sqcup f_n^\prime$ such that for every $i$, $f_i(S^2)=f_i^\prime(S^2)$, but $f$ and $f^\prime$ are not link homotopic.
\end{theo}
This is done by precomposing the link map in Theorem \ref{introthm} with a map which reflects the second sphere, reversing the orientation on the second sphere. Both maps evaluate to distinct elements of $\widetilde{K}$.

This new invariant fits into a commutative diagram for each $i\in \{1,2,3\}$
 \[
 \begin{tikzcd}
 \LM_{2,2,2}\arrow[d, "\widetilde{\sigma}^3"]\arrow[r, "i"]&\LM_{2,2}^4\arrow[d, "\sigma"]\\
\widetilde{K}\arrow[r, "p_i"]&\left(\mathbb{Z}\left[\mathbb{Z}\right]\right)^2\text{,}
 \end{tikzcd}
 \]
 where $i$ is the map which forgets the $i$th component; $\sigma$ is the Kirk invariant and the lower horizontal map, $p_i$, projects onto the other two components, and sets the $i$th meridian to $1$. Hence, $\widetilde{\sigma}^3$ also contains all the information of the Kirk invariant for each two-component sublink. 

Since $\widetilde{\sigma}^3$ takes values in an orbit space, it can be difficult to tell whether one has two representatives of the same value in $\widetilde{K}$ or whether you have two distinct values. To resolve this, we construct invariants to help differentiate the different orbits of  $\widetilde{K}$.
Consider
\[
v=\bigg(\sum_{(i,j,k)\in\mathbb{Z}^3}a^v_{ijk}y^is^jz^k,\, \sum_{(i,j,k)\in\mathbb{Z}^3}b^v_{ijk}z^it^jx^k,\,\sum_{(i,j,k)\in\mathbb{Z}^3}c^v_{ijk}x^iu^jy^k\bigg)
\]
and
\[
w=\bigg(\sum_{(i,j,k)\in\mathbb{Z}^3}a^w_{ijk}y^is^jz^k,\, \sum_{(i,j,k)\in\mathbb{Z}^3}b^w_{ijk}z^it^jx^k,\,\sum_{(i,j,k)\in\mathbb{Z}^3}c^w_{ijk}x^iu^jy^k\bigg)\\
\]
where we have used a normal form of $F/F_3$ to describe each element of $\mathbb{Z}\left[F/F_3\right]$.

Then we show the following theorem holds.
\begin{theo}\label{introthm2}
Suppose $v$ and $w$ represent the same element of $\widetilde{K}$. Then for each $a_{ijk}^v$ there exists an $a^w_{lmp}$ such that $a^v_{ijk}=a^w_{lmp}$, $i=l$, $k=p$ and 
\[
j\equiv m\mod \gcd(i,k)\text{.}
\]
\end{theo}
\begin{exam}
Let
 \[
 v=\left(z^2s^2+z^{-2}s^{-2}-4zs-4z^{-1}s^{-1}+6,\,0,\, 4xu+4x^{-1}u^{-1}-x^2u^2-x^{-2}u^{-2}-6\right)
 \]
 and 
 \[
 w=\left(z^2s+z^{-2}s-4zs-4z^{-1}s^{-1}+6,\, 0, z^2u+z^{-2}u-4zu-4z^{-1}u^{-1}+6\,\right)\text{.}
 \]
Consider the term $a_{022}^v=1$. In $w$ there is no term $a^w_{ijk}$ such that $i=0, k=2, j\equiv 0 \mod \gcd(i,k)$ and $a_{ijk}^w=1$. Hence, $v$ and $w$ are distinct elements in $\widetilde{K}$. 
\end{exam}
In the above example $v$  is in the image of $\widetilde{\sigma}^3$ but we were unable to determine whether $w$ is.

 This kind of indeterminacy is directly analogous to the triple linking number in the classical dimension. Next we go further and extract a stronger invariant. The naive approach would be to create an invariant listing the values of $k\mod\gcd(i,j)$ for each non-zero $a_{ijk}\neq0$. However, the power of the commutator term changes simultaneously for all group elements, so we construct an invariant based on the total Milnor quotient in \cite{MP},
\[
\overline{\mu}:\widetilde{K}\rightarrow \mathcal{A},
\]
where $\mathcal{A}$ is a space of affine points, lines, and planes which measures how all the commutator terms change when we change basing paths. We show this invariant can be used to distinguish two link maps. See Section \ref{newinv} for details on how $\mathcal{A}$ is defined.
\begin{exam}
Let
\[
v=(zs-z+z^{-1}s^{-1}-z^{-1}, 0, x-xu+x^{-1}-x^{-1}u^{-1})
\]
and
\[
w=(zs^2-z+z^{-1}s^{-2}-z^{-1}, 0, x-xu^2+x^{-1}-x^{-1}u^{-2})\text{.}
\]
We cannot use Theorem \ref{introthm2} to distinguish between $v$ and $w$, both of which are realised by link maps. However, we show that $\overline{\mu}$ can detect that these maps are different in Section \ref{newinv}.
\end{exam}

We now consider $n$-component annular link maps, which by definition are immersions 
\[
\coprod_{i=1}^n S^1\times I \rightarrow B^3\times I,
\]
such that restricting to the boundary $\coprod_{i=1}^n S^1\times\{j\} \rightarrow B^3\times \{j\}$ is an embedding of $n$-component unlink of circles in $B^3$ when $j=\{0,1\}$ where the $i$th component of maps onto the $i$th component of the unlink for each embedding. We consider the set of three-component annular link maps up to link homotopy, which we denote by $\SL_{2,2,2}^4$. Using $\widetilde{\sigma}^3$ we prove the following proposition.

\begin{prop}
The group of $n$-component annular link map up to link homotopy is non-abelian when $n\geq 3$.
\end{prop}
This result was proven in \cite{M21} by studying the isomorphisms of the free Milnor group, which is isomorphic to the group of embedded annular link maps up to link homotopy, whereas our proof uses the self-intersections of annular link maps to prove the result.

The benefit of studying annular link maps instead of link maps is that the set of annular link maps up to link homotopy is a group under a natural stacking operation. This is a similar move Habegger and Lin made in \cite{HL90} which allowed them to classify $n$-component classical links up to link homotopy.
Combining our approach of studying intersections with Meilhan and Yasuhara's approach of studying the induced isomorphism, we construct an invariant of three-component annular link maps up to link homotopy giving a group homomorphism
\[
\Theta:\SL_{2,2,2}^4\rightarrow \left(\mathbb{Z}\left[F/F_3\right]\right)^3 \rtimes \Aut(MF(3)),
\]
where the semi-direct product records the self-intersection information on each component and the automorphism of the free Milnor group on three-generators, $\Aut(MF(3))$, which is induced by the annular link map. This map is an extension of the work by Meilhan and Yasuhara in \cite{M21} to the non-embedded case. Note that since annular link maps have a canonical choice of meridians and basing there is less indeterminacy in this invariant compared to $\widetilde{\sigma}_3$.

We show that $\Theta$ is a group homomorphism and there exists a closure map $\SL_{2,2,2}^4\rightarrow \LM_{2,2,2}^4$ and establish the following theorem.

\begin{theo}
The following diagram commutes
\[
\begin{tikzcd}
\SL_{2,2,2}^4\arrow[d, "\Theta"]\arrow[r,two heads]&\LM_{2,2,2}^4\arrow[d, "\widetilde{\sigma}^3"]\arrow[r,two heads, "i"]&\LM_{2,2}^4\arrow[d, "\sigma"]\\
\left(\mathbb{Z}\left[F/F_3\right]\right)^3 \rtimes \Aut(MF(3))\arrow[r]& \widetilde{K}\arrow[r, "p_i"]&\left(\mathbb{Z}\left[\mathbb{Z}\right]\right)^2\text{.}
\end{tikzcd}
\]

\end{theo}

We show that the subgroup of annular link maps whose closure is a Brunnian link map -- link maps where each two component sublink is trivial up to link homotopy -- is a normal subgroup, containing the subgroup of annular link maps link homotopic to a embedded annular link map.  

Using the work of Fenn and Rolfsen \cite{FR}, Brendle and Hatcher \cite{HB}, and Benjamin Audoux, Paolo Bellingeri, Jean-Baptiste Meilhan, and Emmanuel Wagner \cite{A}, we show how to construct a class of link maps and compute $\widetilde{\sigma}^3$ on this class in terms of the Kirk invariants of the two-component sublinks. From this we prove the following proposition.
\begin{prop}\label{introconnect}
Connect sum on the set of $n$-component link maps does not give a well defined group structure for $n\geq 3$.
\end{prop}

This paper has been extracted from my 2022 PhD thesis at Durham University.

\subsection{Organisation of paper}

In Section \ref{chap1}, we define the equivariant intersection form, geometric intersections invariant, and establish their relationship to each other.

In Section \ref{chap2}, three types of homotopy are defined which allow us to understand link homotopy in terms of a more rigid structure. We also discuss the results which have been achieved in the two-component case.

In Section \ref{chap4}, Milnor groups are defined and we revise commutator calculus. We then show that the Milnor group of the free group on two generators, $F/F_3$, is isomorphic to the integral Heisenberg group of matrices and compute its group homology. We then calculate the homology of the complement of generically immersed two-spheres in the four-sphere and show that the second homology is generated by Clifford tori around each self-intersection. Using Dwyer's theorem we show that the Milnor group of the fundamental group of the complement of two generically immersed disjoint spheres is isomorphic to $F/F_3$. 

In Section \ref{chap5}, we give a link homotopy invariant for based three-component link maps of oriented two-spheres. We then remove the based restriction to create $\widetilde{\sigma}^3$.  We then develop tools for being able to differentiate between values of $\widetilde{\sigma}^3$ since our invariant takes values in a quotient space. We show there exist examples of link maps with the same image in $S^4$ but are distinct. Additionally, we develop an invariant which highlights how analogous our invariant is to the triple linking number. We then construct an invariant which measures the size of the values of our invariant.

In Section \ref{chap6}, we construct an invariant for annular link maps and show how one can use annular link maps to create link maps and establish the existence of an annular link map invariant which records induced automorphism information and the intersection information.

In Section \ref{chap7}, we discuss generalising $\widetilde{\sigma}^3$ to establish an $n$-component link map invariant.

\subsection*{Acknowledgments}I would like to thank my PhD supervisor Mark Powell for his guidance during my PhD and advice in crafting this paper. I also wish to thank my external examiner Rob Schneidermann for his helpful suggestions.

\section{Algebraic and geometric intersections}\label{chap1}

\subsection{The intersection form}
Let $X$ be a CW complex and  $\widetilde{X}$ be the regular cover of $X$, associated to the kernel of a surjective group homomorphism  $\phi:\pi_1(X)\rightarrow \pi$. Let $Y\subset X$ and define $\widetilde{Y}$ to be the pre-image of $Y$ under the covering map.  Using deck transformations, $C_*(X, Y;\Lambda)\colon=C_*(\widetilde{X},\widetilde{Y})$ is given a left $\Lambda$-module structure, where $\Lambda:=\mathbb{Z}\pi$; where $\Lambda$ has a involution coming from extending the involution $\overline{g}:=g^{-1}$ on $\pi$ linearly. With homology denoted by $H_*(X,Y;\Lambda)$.

We define the cochain complex by 
\[
C^*(X, Y; \Lambda)\colon= \Hom_{\rit-\Lambda}\left(\overline{C_*(X, Y;\Lambda)}, \Lambda\right),
\]
where $\overline{C_*(X, Y;\Lambda)}$ is the involuted chain complex, with corresponding cohomology denoted $H^*(X, Y; \Lambda)$. This is a left $\Lambda$-module with the action of $\Lambda$ defined by
\[
(r\cdot f)(\sigma)\colon=rf(\sigma),
\]
where $r\in \Lambda$, $f\in C^*(X, Y; \Lambda)$, and $\sigma\in\overline{C_*(X, Y;\Lambda)}$.

\begin{defn}
Define $\kappa:C^n(X, Y; \Lambda)\rightarrow \overline{\Hom_{\lef-\Lambda}\left(C_n(X,Y;\Lambda), \Lambda)\right)}$ to be the map
\[
f\mapsto \left( \sigma \mapsto \overline{f(\sigma)}\right)\text{.}
\]
\end{defn}

We denote the evaluation map by 
\[
p:H^n\left(\overline{\Hom_{\lef-\Lambda}\left(C_*(X,Y;\Lambda), \Lambda\right)}\right)\rightarrow\overline{\Hom_{\lef-\Lambda}\left(H_n(X,Y;\Lambda\right), \Lambda)}
\]
and define $\ev\colon:=p\circ \kappa$. We often refer to $\ev$ as the evaluation map, however we will be specify which map if it is unclear.

Let $W$ be a compact orientable connected four-manifold with $\partial W$ which is equipped with Poincar\'e duality maps
\[
PD:H^n(W;\Lambda)\rightarrow H_{4-n}(W, \partial W;\Lambda)
\text { and }PD:H^n(W, \partial W;\Lambda)\rightarrow H_{4-n}(W;\Lambda)\text{.}
\]
Consider the following composition:
\[
\Phi:H_{2}(W;\Lambda)\xrightarrow{q} H_{2}(W, \partial W;\Lambda)\xrightarrow{\PD^{-1}}H^2(W;\Lambda)\xrightarrow{\ev}\overline{\Hom_{\lef-\Lambda}\left(H_2(W;\Lambda), \Lambda\right)}\text{.}
\]

\begin{defn}
The \emph{equivariant intersection form} on $W$ is the map
\begin{align*}
H_2(W;\Lambda)\times H_2(W;\Lambda)&\rightarrow \Lambda\\
(a,b)&\mapsto\Phi(b)(a)\text{.}
\end{align*}
\end{defn}

From the definition, it is clear this form is sesquilinear and it is well known that the intersection form is hermitian \cite{supsup}.

\subsection{Geometric Intersections}
We define geometric intersections of $2$-spheres in an oriented four-manifold $M$. 
 
Fix an orientation of $S^2$ with a basepoint $x_0\in S^2$. Let $\left(f, \gamma_f\right):S^2\rightarrow M$ be a based map of a sphere i.e a map paired with a path $\gamma_f:\left[0,1\right]\rightarrow M$ where $\gamma(0)=m_0$ and $\gamma(1)=f\left(x_0\right)$. We often suppress the $\gamma_f$ in our notation and instead just write $f:S^2\rightarrow M$ to be a based map, where it is implicit there exists a specific $\gamma_f$. Let $g:S^2\rightarrow M$ be another based map. Assume $f$ and $g$ are generic transverse immersions and further assume the map $f\cup g$ is generic also.  By compactness there exists a finite number of intersections between the images of $f$ and $g$. Suppose $x$ is an intersection of $f$ and $g$ where $p_1\in f^{-1}\left(x\right)$ and $p_2\in g^{-1}\left(x\right)$. Let $k_x:\left[0,1\right]\rightarrow S^2$ be a path such that $k_x(0)=x_0$ and $k_x\left(1\right)=p_1$. Similarly, define $l_x:\left[0,1\right]\rightarrow S^2$ with $l_x(0)=x_0$ and $l_x(1)=p_2$. We define
\[
g_x=\gamma_f\cdot f\left(k_x\right)\cdot g\left(\overline{l}_x\right)\cdot  \overline{\gamma}_g\in \pi_1(M)
\]

\begin{defn}\normalfont
We define the \emph{geometric intersection} to be
\[
\lambda_{\text{geo}}(f, g)=\sum_{x\in f(S^2)\cap g(S^2)}\varepsilon_xg_x\in\mathbb{Z}\pi,
\]
where $\varepsilon_x$ is the sign of the intersection at $x$.
\end{defn}
The geometric intersection is independent of the choice of paths on each sphere as $S^2$ is simply connected. However, if we consider $f$ with a different choice of basing path this may change the value of the geometric intersection. Hence, it is only well defined on based maps. The geometric intersection form agrees with the equivariant intersection form \cite{AR}. Hence, we will use $\lambda$ to refer to both the geometric and the equivariant intersection forms.

Let $f^+\colon S^2\rightarrow M$ be a normal push-off, using a section of the normal bundle of $f$ which is transverse to the $0$-section. To make sense of $\lambda_{geo}(f, f)$ we  define:
\[
\lambda_{\text{geo}}(f,f):=\lambda(f, f^+).
\]
This is independent of the choice of normal push off.

We define a self-intersection number related to $\lambda$. Let $I$ be the set of all double points of $f$ and let $p\in I$. Let $p_1, p_2\in f^{-1}\left(p\right)$ with $p_1\neq p_2$ and let $\delta_1,\delta_2:[0,1]\rightarrow S^2$ be paths from $x_0$ to $p_1$ and $p_2$ respectively. Set
\[
g_p=\gamma_f\cdot f\left(\delta_1\right)\cdot f\left(\bar{\delta}_2\right)\cdot\bar{\gamma}_f\in\pi_1(M)\text{.}
\]

We could have swapped the roles of $\delta_1$ and $\delta_2$ and this would give us the loop $g^{-1}_p$. To resolve this we take values in the quotient $\mathbb{Z}\widetilde{\pi}\colon=\mathbb{Z}\pi/\{g\sim g^{-1}\}$. This is a group quotient where we view $\mathbb{Z}\pi$ as an abelian group.
\begin{defn}\normalfont
The \emph{self-intersection number} of a generic (based) immersion $f:S^2\rightarrow M$ is defined to be 
\[
\mu(f)\colon=\sum_{p\in I}\varepsilon_pg_p\in \mathbb{Z}\widetilde{\pi}\text{,}
\]
where $\varepsilon_p$ is the sign of the intersection at $p$.
\end{defn}
We can relate the self-intersection number and geometric intersection number by the following.
\begin{theo}\label{intersection}
Let $f:S^2\rightarrow M$ be a generic immersion and let $\iota:\mathbb{Z}\rightarrow \mathbb{Z}\pi$ be the ring homomorphism where $n\mapsto n1$. Then the following equation holds
\begin{equation}
\lambda\left(f, f\right):=\mu(f)+\overline{\mu(f)}+\iota\left(\chi(f)\right)\in\mathbb{Z}\pi,
\end{equation}
where $\overline{\mu(f)}$ is given by taking the representatives we use to describe $\mu(f)$ and sending each group element to its inverse, and $\chi(f)$ is the Euler number of the normal bundle of $f$.
\end{theo}
It is easily shown from this that  $\mu$ is not a homotopy invariant. Cusp homotopies (defined in the next section) change the Euler number of the normal bundle by $\pm 2$ but $\lambda$ remains fixed as it is a homotopy invariant. Therefore must $\mu$ must change for $\lambda$ to remain constant.

\section{Homotopy of surfaces and the Kirk invariant}\label{chap2}

\subsection{The structure of link homotopy in dimension four}
\begin{defn}
A map
\[
f=f_1\sqcup\ldots\sqcup f_n:\coprod_{i=1}^nS^2_i \rightarrow S^4,
\]
where $S^2_i\cong S^2$, is called a \emph{link map} if components disjoint in the domain are disjoint in the image, i.e $f_i(S^2_i)\cap f_j(S^2_j)=\phi$ for $i\neq j$. A \emph{link homotopy} is a homotopy through link maps. We denote the set of $n$-component link maps up to link homotopy by $\LM_{\underbrace{2,\ldots, 2}_{n}}^4$. If each of our spheres comes equipped with a basing path from the basepoint in $S^4$ to $f_i(S^2_i)$ we call $f$ a \emph{based link map}. The set of based link maps up to link homotopy is denoted $\LM_{\underbrace{2,\ldots, 2}_{n}*}^4$. 
\end{defn}

\begin{prop}[\cite{FQ}]\label{FQ}
A generic (link-)homotopy between two generic immersions is (link-)homotopic to concatenations of isotopies, finger moves, Whitney moves, and cusp homotopies.
\end{prop}

Hence, we can think of a link homotopy between two link maps as a combination of isotopy, Whitney moves, finger moves and cusp homotopies.

\subsection{Two-component link maps and the Kirk invariant}
For this subsection all link maps will have two components unless stated otherwise. In 1988, Koschorke showed that, up to link homotopy, connect sum is  well defined on two-component link maps \cite{K}. This turns $\LM_{2,2}^4$ an abelian monoid. In 1999 Bartels and Teichner showed that each element in $\LM_{2,2}^4$ has an inverse, establishing that $\LM_{2,2}^4$ is a group \cite{BT}. Two-component link maps have an invariant defined by Paul Kirk in 1988, which we now call the Kirk invariant. In 2018, Schneiderman and Teichner showed that the Kirk invariant is a complete invariant of $\LM_{2,2}^4$ \cite{ST}. We define this invariant since this is the inspiration for our three-component invariant.

Let $f=f_1\sqcup f_2:S^2_1\coprod S^2_2\rightarrow S^4$ be a link map and define
\[
M_i=S^4\setminus \nu(f_j),
\]
where $i\neq j$ and $\nu(f_j)$ is the regular neighbourhood of $f_j$. 

We compute intersections $\lambda(f_1, f_1)\in \mathbb{Z}\pi_1(M_1)$. This is not a link homotopy invariant since the fundamental group of the complement of $f_2$ changes over a generic link homotopy. To fix this we use homology; since the homology of the complement of a generic sphere in $S^4$ is constant throughout a generic link homotopy. Define $\lambda(f_1, f_1)\in \mathbb{Z}H_1(M_1)$  to be $\sigma_1(f)$. Swapping the roles of $f_1$ and $f_2$ we get another invariant of $f$ called $\sigma_2(f)$.

\begin{defn}
The \emph{Kirk invariant} of a link map $f$ is defined to be 
\[
\sigma(f)=\left(\sigma_1(f), \sigma_2(f)\right)\in\mathbb{Z}[\mathbb{Z}]\oplus\mathbb{Z}[\mathbb{Z}]\text{.}
\]
\end{defn}
We can realise the image of the Kirk invariant \cite{ST} using a technique Fenn and Rolfsen created in their paper to construct the first example of a link homotopically non-trivial link map in dimension two \cite{FR}. 

\begin{con}[The Jin-Kirk construction]
Consider a two-component link, $L=L^1\coprod L^2\subset S^3\times\big\{\frac{1}{2}\big\}\subset S^3\times I$, where $I=\left[0,1\right]$. Suppose the linking number of $L$ is zero; both components null-homotopic in the complement of the other and both components unknotted. Consider the trace of a homotopy inside $S^3\times \left[\frac{1}{2}, 1\right]$ which fixes the second component and does a null homotopy of the first. Then take the trace of a homotopy running backwards from $S^3\times \{\frac{1}{2}\}$ to $S^3\times \{0\}$ with the first component fixed and a null-homotopy applied to the second component. Capping off both ends of the cylinder with a $D^4$ and both ends traces of the link with a $D^2$ gives a link map in $S^4$.
\end{con}

The construction of the link map above is, up to link homotopy, independent of the choice of null homotopy on each component as $\pi_2\left(S^3\setminus L^i\right)=0$, where $L^i$ is a component of the link we started the construction with. This gives a surjective map \cite{ST}
\[
JK:\mathfrak{L}\twoheadrightarrow \LM_{2,2}^4\text{,}
\]
where $\mathfrak{L}$ is the set of two-component links with each component null-homotopic in the complement of the other and both components unknotted.

\begin{figure}
    \centering
    \includegraphics[scale=0.75]{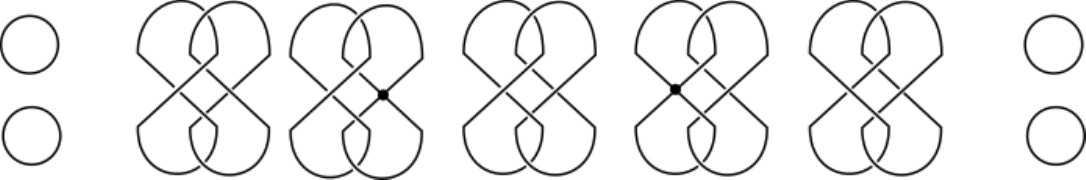}
    \caption{The JK-construction applied to the Whitehead link with time moving from left to right with a single intersection on each sphere.}
    \label{JKW}
\end{figure}
\begin{exam}\normalfont
Let $FR$ be the link map created from the JK construction described in Figure \ref{JKW}. Then
\[
\sigma(FR)=\left(t+t^{-1}-2, 2-t-t^{-1}\right)\text{.}
\]
\end{exam}

Schniederman and Teichner proved the following result which classified two-component link maps.
\begin{theo}\label{clasif}
Let $z=2-t-t^{-1}$ and identify $\mathbb{Z}=z\cdot\mathbb{Z}\left[z\right]/z^2\cdot\mathbb{Z}\left[z\right]$. Then the sequence
\[
\begin{tikzcd}
0\arrow[r]&\LM_{2,2}^4\arrow[r, "\sigma"]&z\cdot\mathbb{Z}\left[z\right]\oplus z\cdot\mathbb{Z}\left[z\right]\arrow[r]&\mathbb{Z}\arrow[r]&0\text{,}
\end{tikzcd}
\]
where the right most map is addition, is exact.
\end{theo}
Schniederman and Teichner show that $\LM_{2,2}^4$ has an interesting module structure.
\begin{theo}\label{2-compmod}
Let $R:=\mathbb{Z}\left[z_1,z_2\right]/\left(z_1z_2\right)$. Then $\LM_{2,2}^4$ is a free $R$-module of rank $1$ and as a $\mathbb{Z}$-module is free of infinite rank.
\end{theo}

\section{The Heisenberg group}\label{chap4}

In this section we develop the  commutator calculus and group theory necessary for our three-component invariant. We discuss Milnor groups which will be necessary for generalising the Kirk invariant since the $n$-component invariant Kirk developed is based in homology which loses too much geometric information due to homology being abelian.

\subsection{Milnor groups and commutator calculus}\label{ff3}
\begin{defn}\normalfont
Let $G$ be a group normally generated by elements $x_1,\ldots, x_n\in G$. Then we call the group
\begin{align*}
MG:=G/\left<\left<\left[x_i, gx_ig^{-1}\right]\right>\right>\ 1\leq i\leq n,\ \forall g\in G\\
\end{align*}
the \emph{Milnor group of $G$}. We call $G$ a \emph{Milnor group} if $G=MG$.
\end{defn}

\begin{rem}
There is still a question about if the choice of normal generators affects the resulting quotient group. However, for any choice of normal generators we will choose $x_1,\ldots, x_n$ for a group $G$, any other  set of normal generators we choose $x_1^\prime,\ldots,x_n^\prime$  will have the property $x_i^\prime=gx_ig^{-1}$ for some $g\in G$. This will guarantee that this quotient is fixed for our purposes as we will be choosing the meridians of a link as our choice of normal generators of the fundamental group of the complement of the link map\footnote{The author tried to find a reference to show that the Milnor quotient is independent of the choice of normal generators and was unable to find one. However, he did find that Freedman and Teichner in \cite{DT}, define Milnor groups relative to a choice of normal generators, this suggest they are either possibly different or it is not known if all such quotients are equal.}.
\end{rem}

Milnor used such groups in the study of classical links up to link homotopy in \cite{LGM}, where he was able to classify three-component links up to link homotopy. We will see that for three-components, one dimension up, we only care about about one particular Milnor group. Before studying this Milnor group we will first discuss commutator calculus and lower central series quotients.

Let $F(n)$ be the free group on $n$ generators. Consider the free group on two generators $F:=F(2)$, and its lower central series quotient $F/F_3$. %The importance of this group to us is perhaps somewhat opaque at the moment but we wish to understand this group.
\begin{lem}
The group $F/F_3$ admits the presentation 
\[
\left<x, y\,|\left[x,\left[x,y\right]\right],\, \left[y,\left[x,y\right]\right]\right>.
\]
\begin{proof}
Let $N$ be the normal subgroup generated by $\left[x, \left[x, y\right]\right]$ and $\left[y,\left[x, y\right]\right]$. It is clear that $N\subseteq F_3$. We wish to show the relations defining $N$ imply that $F_3\subseteq N$. First, let us prove that for any $\omega\in F$ we have
\begin{equation}\label{F/F_3cong}
\left[\omega, \left[x,y\right]\right]\equiv 1 \text{ mod } N\text{.}
\end{equation}
We do this by induction of the length of the word $\omega$. It is clear in the cases where $\omega$ is equal to $x$ or $y$. Checking their inverses we have
\[
\left(x(\left[x^{-1}, \left[x,y\right]\right]x^{-1}\right)^{-1}=\left[x,\left[x,y\right]\right].
\]
Hence
\[
\left[x^{-1}, \left[x,y\right]\right]\equiv 1 \text{ mod } N.
\]
The case for $y^{-1}$ is similar.
Now assume that the result holds true for words of length $n-1$. Let $\omega$ be a word of length $n$. Now $\omega=g\omega^\prime$ where $g$ is equal to $x$, $y$ $x^{-1}$ or $y^{-1}$ and $\omega^\prime$ is a word of length $n-1$. Assume that $g=x$; the proofs of the other cases are similar. Using results from \cite{CGT} and the inductive hypothesis we have
\begin{align*}
\left[x\omega^\prime, \left[x, y\right]\right]&=\left[x,\left[\omega^\prime, \left[x,y\right]\right]\right]\left[\omega^\prime, \left[x, y\right]\right]\left[x,\left[x,y\right]\right]\\
&\equiv 1 \text{ mod } N  
\end{align*}

Every element of $F_2$ is a product of $\left[x,y\right]$, its inverse, and their conjugates. So any element can be expressed as a product of these commutators and as such they can be expressed as some minimum length as a product of commutators. We will prove that $\left[\omega, K\right]\in N$, where $K$ is some product of commutators, by induction. The base case is done by checking that all commutators of the form $h\left[x, y\right]\bar{h}$, for some $h\in F$. Applying congruence \eqref{F/F_3cong}, we have \[
\left[\omega, h\left[x,y\right]h^{-1}\right]=h\left[h^{-1}\omega h,\left[x, y\right]\right]h^{-1}\equiv 1 \text{ mod } N.
\]
Similarly, 
\[
\left[\omega, h\left[y,x\right]h^{-1}\right]=h\left[h^{-1}\omega h,\left[y, x\right]\right]h^{-1}\equiv 1 \text{ mod } N.
\]
Assume this holds true for any element in $F_2$ that can be written as product of $n-1$ $\left[x,y\right]$, its inverse, and their conjugates. Let $K$ be an element of $F_2$ which can be written $n$ elements of such commutators. Then we have $K=LK^\prime$ where $L$ is equal to $h\left[x,y\right]h^{-1}$ or $h\left[y,x\right]h^{-1}$. Again assume $L=h\left[x,y\right]h^{-1}$; the proof of the other case is similar. Using results from \cite{CGT} and the inductive hypothesis we have
\[
\left[\omega, LK\right]=\left[\omega, L\right]\left[L,\left[\omega, K^\prime\right]\right]\left[\omega,K^\prime\right]
\equiv 1 \text{ mod } N\text{.}
\]
Hence, $F_3\subseteq N$ as required.
\end{proof}
\end{lem}

\begin{lem}\label{MF3}
 The groups $MF$ and $F/F_3$ are isomorphic.
 \begin{proof}
 Let $N$ be the normal subgroup we quotient $F$ with to get $MF$. We have the following equivalence 
 \[
 \left(\left(\left[x, \left[x,y\right]\right]\right)^{\left[y,x\right]}\right)^{-1}=\left[x, yxy^{-1}\right]\text{.}
\]
 Thus, adding the $\left[x, \left[x, y\right]\right]=1$ relation is equivalent to adding the relation $\left[x, yxy^{-1}\right]=1$. Similarly, adding the relation $\left[y, \left[y, x\right]\right]=1$ is equivalent to $\left[y, xyx^{-1}\right] =1$. Thus, we have $F_3\subseteq N$.  We need to show $\left[x, \omega x\omega^{-1}\right]=1$ and $\left[y, \omega y \omega^{-1}\right]=1$ hold for any reduced word $\omega$ when the relations for $F/F_3$ are added. We will do this by induction on the word length of $\omega$ where the base case is the relations given by $F/F_3$. Let $k$ be one of the generators $x$, $y$ or their inverses. Let $\omega=k\omega^\prime$, where $\omega$ and $\omega^\prime$ are freely reduced words of length $k$ and $k-1$ respectively. Using commutator calculus results in \cite{CGT} we have
\[
 \left[x, \omega x\omega^{-1}\right]=\left[x, k\right]\left[k, \left[x, \omega^\prime x\omega^{\prime-1}\right]\right]\left[x, k^{-1}\right]\text{ mod } F_3 \\
\]
By the inductive hypothesis the central term vanishes and we have
\begin{align*}
\left[x, \omega x\omega^{-1}\right]&\equiv \left[x, k\right]\left[x, k^{-1}\right]\text{ mod } F_3\\
&\equiv\left[x,k\right]\left[k, x\right]\text{ mod } F_3\\
&\equiv 1 \text{ mod } F_3
\end{align*}
Since $F_3\subseteq N$, these relations still hold mod $N$ and thus $N=F_3$.
  \end{proof}
 \end{lem}
\subsection{The Heisenberg group and its homology}
 We  show that $F/F_3$ is isomorphic to a group of integral matrices and use this to to compute the homology of the group. First, we define what we mean by the homology of a group.

\begin{defn}\normalfont
Let $G$ be a group. Then the group homology and cohomology of $G$ are defined to be 
\[
H_*(G;\mathbb{Z}):=H_*(\textbf{B}G;\mathbb{Z}),\,\,\,\,\,\
H^*(G;\mathbb{Z})=H^*(\textbf{B}G; \mathbb{Z}),
\]
where $\textbf{B}G$ is the classifying space.
\end{defn}

 We will work exclusively with discrete finitely presented groups. Their classifying spaces are Eilenberg-Maclane spaces, $K(G, 1)$ i.e. connected CW complexes with fundamental group equal to $G$ and all higher homotopy groups vanishing. 
 
 \begin{defn}\normalfont
 We call the group of matrices 
 \[
 H=\left\{\begin{pmatrix}
 1&a &c\\
 0&1& b\\
 0&0&1
 \end{pmatrix}\bigg| a,b,c\in\mathbb{Z}\right\}\subseteq SL(3,\mathbb{Z})\text{,}
 \]
the (integral) Heisenberg group.
 \end{defn}

 We will show that both $F/F_3$ and $H$ are isomorphic by making use of a normal form for a presentation of $F/F_3$.
 
 \begin{defn}\normalfont
 Let $F(n)$ be the free group on $n$ generators and $G$ be a group such that there exist a surjective homomorphism $\pi:F(n)\rightarrow G$. Then a \emph{normal form} is an injective map $s:G\rightarrow F(n)$ such that $\pi\circ s=I_G$.
 \end{defn}
We will first show the following result. 
\begin{lem}\label{mat}
 Let  
 \[
 X=\begin{pmatrix}
 1&0&0\\
 0&1&1\\
 0&0&1
 \end{pmatrix},\,\,
 Y=\begin{pmatrix}
 1&1&0\\
 0&1&0\\
 0&0&1
 \end{pmatrix},\,\,
 Z=\begin{pmatrix}
 1&0&-1\\
 0&1&0\\
 0&0&1
 \end{pmatrix}\in H.
 \]
 Then
 \[
 \begin{pmatrix}
 1&a &-c\\
 0&1& b\\
 0&0&1
 \end{pmatrix}=X^aZ^cY^b.
 \]
 \begin{proof}
 This can be checked directly by matrix multiplication.
 \end{proof}
 \end{lem}

 \begin{prop}
 The Heisenberg group and $F/F_3$ are isomorphic.
 \begin{proof}

 Consider $F/F_3$ with the presentation 
 \[
 \left<x, y, z|\left[x,z\right],\,\left[y,z\right],\,\left[x,y\right]=z\right>.
 \]
 Let $\omega$ be a fully reduced word in $x$, $y$ and $z$. Using the relations we can collect all of the $z$ or $z^{-1}$ to the back resulting in $\omega=\omega^\prime z^t$, where $\omega^\prime$  is some word in $x$ and $y$. We can move each $y$ and $y^{-1}$ to the back by passing using the relations, introducing more $z$ and $z^{-1}$. Eventually, we have the word 
 \[
 x^uz^wy^u=\omega\in F/F_3
 \]
 for some $u, v, w\in\mathbb{Z}$. Define a homomorphism $\phi:F/F_3\rightarrow H$ by specifying that $\phi(x)=X$, $\phi(y)=Y$ and $\phi(z)=Z$. This is well defined since 
 \[
 \phi\left(\left[x,z\right]\right)=\phi\left(\left[y, z\right]\right)=\phi\left(\left[x,y\right]z^{-1}\right)=1
 \]
 and is clearly surjective so all that remains to be proven is injectivity. Suppose the map failed to be injective then we would be able to find $a$,$b$, $c$, $a^\prime$, $b^\prime$ $c^\prime\in\mathbb{Z}$ such that $x^az^cy^b\neq x^{a^\prime}z^{c^\prime}y^{b^\prime}$ and
 \[
 \phi\left(x^az^cy^b\right)=\phi\left(x^{a^\prime}z^{c^\prime}y^{b^\prime}\right),
 \]
 and by Lemma \ref{mat} we have
 \[
  \begin{pmatrix}
 1&a &-c\\
 0&1& b\\
 0&0&1
 \end{pmatrix}= \begin{pmatrix}
 1&a^\prime &-c^\prime\\
 0&1& b^\prime\\
 0&0&1
 \end{pmatrix}.
 \]
 Hence we have $a=a^\prime$, $b=b^\prime$ and $c=c^\prime$. This contradicts $x^az^cy^b\neq x^{a^\prime}z^{c^\prime}y^{b^\prime}$ and completes the proof.  
 \end{proof}
 \end{prop}
 This establishes that a normal form for $F/F_3$
\begin{lem}\label{normgen}
Let $S$ be the image of a generic $n$-component link map and $x_i$ be a meridian of the $i$th component. Then $x_1,\ldots x_n$ is a set of normal generators for the fundamental group of $S^4\setminus S$.
\begin{proof}
Define $X_S:=S^4\setminus \nu(S)$. To prove the result, we will give a Kirby diagram description of the normal bundle of $S$ and turn the handle decomposition upside down and glue to the boundary of $X_S$ to make $S^4$; from this it will be clear that we have a set of normal generators.
Consider the normal neighbourhood of the $i$th sphere (a description of the normal bundle is given in \cite[Section 6.1]{SG}). A Kirby diagram of a normal neighbourhood of an immersion is given by a single $0$-handle, a single $2$-handle, and $k$ many $1$-handles (one for each intersection on the immersed surface). Turning this upside down we have a single $2$-handle, $k$ many $3$-handles, and a single $4$-handle. We now attach the $2$-handle to the $i$th boundary component, where attaching sphere of the dual two handle is a meridian of the $i$th sphere. To see the $2$-handle attaches to the meridian of the sphere, notice that the $2$-handle of the plumbing is the thickened $2$-cell of the sphere and thus the dual handle from turning upside down the new attaching sphere is the meridian of the sphere.  The fundamental group of this new space once we have attched the two handle is $\pi_1(S^4\setminus \nu(S))/\left<\left<x_i\right>\right>$. Doing this for each sphere we recover $S^4$ which has trivial fundamental group. Hence the meridians are a set of normal generators for the complement as required.
\end{proof}
\end{lem}

\begin{lem}\label{fmr}
Let $f$ be a link map and a generic immersion. We can do finger moves on $f$ to get another link map $\tilde{f}$ such that 
\[
\pi_1\left(S^4\setminus \nu(\tilde{f})\right)\cong M\pi_1\left(S^4\setminus \nu(f)\right)\text{.}
\]
\begin{proof}
We can do self-finger moves on each component, using a finger which misses the other components, to introduce relations of the form $\left[\alpha, \beta\alpha\beta^{-1}\right]$, where $\alpha, \beta\in\pi_1(S^4\setminus \nu(f))$. Since Milnor groups are nilpotent and normally generated by finitely many elements, by Theorem  \ref{normgen}, they are finitely presented \cite{BT}. Hence, we only need to do finitely many to make the complement isomorphic to a Milnor group.
\end{proof}
\end{lem}

\begin{prop}\label{milncomp}
Let $f$ be a generic two-component link map then
\[
M\pi_1(S^4\setminus \nu(f))\cong F/F_3
\]
\end{prop}
\begin{proof}
The proof of this result can be found in \cite{Mythesis} and consists of a Mayer-Vietoris computation.
\end{proof} 

\begin{cor}
Let $f=f_1\sqcup f_2\sqcup f_3:S^2\coprod S^2\coprod S^2\rightarrow S^4$ be a generic link map. Then we can apply finger moves to $f$ to create a link map $\widetilde{f}_1\sqcup \widetilde{f}_2\sqcup\widetilde{f}_3:S^2\coprod S^2\coprod S^2\rightarrow S^4$ such that 
\[
\pi_1(S^4\setminus \nu(\widetilde{f}_i\cup \widetilde{f}_j))\cong F/F_3,
\]
where $i\neq j$.
\begin{proof}
By Lemma \ref{fmr}, finger moves can be applied to $f$ such that the complement of any two distinct components is a Milnor group. By Proposition \ref{milncomp}, this Milnor group is isomorphic to $F/F_3$.
\end{proof}
\end{cor}

\section{The three-component link map invariant}\label{chap5}

We now define our three-component invariant. We first do this in the case of based link maps and then explain how to make an invariant for unbased link maps.

\subsection{Based link maps}
Firstly, we choose the induced orientation on $S^n$ from the standard orientation on $\mathbb{R}^{n+1}$. Fix a basepoint $s_0\in S^4$ and
consider a generically immersed, based, link map $f=f_x\sqcup f_y\sqcup f_z:S^2_x\coprod S^2_y \coprod S_z^2\rightarrow S^4$, where the basing paths for each sphere are given by $\gamma^f_x$, $\gamma_y^f$ and $\gamma_z^f$ respectively. Homotop $\gamma_x^f$, fixing the end points, such that on $\left[1-\varepsilon, 1\right]$, for some $\varepsilon>0$, $\gamma$ lies in the fibre of the normal bundle above the point $\gamma_x^f(1)$. Let $U$ be an open subset of $f_x$, around $\gamma_x^f(1)$ which contains no intersections and admits a trivialisation of the normal bundle, $\phi:U\times D^2\rightarrow \nu(f_x)|_U$. We define a meridian by the following concatenation of paths: travel along the basepoint $\gamma_x^f$ from $x_0$ to $\gamma_x^f\left(1-\varepsilon\right)$, then travels along the generator of $\pi_1\left(U\times D^2\setminus U\times\{0\}\right)$ which has linking number $+1$ with $f_x$, and then travels back along $\gamma_x^f$ from $\gamma_x^f\left(1-\varepsilon\right)$ to the $x_0$. Call this meridian $x$. Define $y$ and $z$ to be similar meridians for the other respective components. These meridians are normal generators of $\pi_1\left(S^4\setminus \nu(f)\right)$ and thus are generators of $M\pi_1\left(S^4\setminus \nu(f)\right)$. Let $A_x, A_y$ and $A_z$ be the normal closures of $x$, $y$ and $z$ respectively, inside $M\pi_1\left(S^4\setminus \nu(f)\right)$. By Proposition \ref{milcomplement} we have
\begin{align*}
M\pi_1\left(S^4\setminus \nu(f)\right)/A_x&\cong\left<y, z, s| \left[y, s\right], \left[z, s\right],\left[y, z\right]s^{-1}\right>\\
M\pi_1\left(S^4\setminus \nu(f)\right)/A_y&\cong\left<z, x, t| \left[z, t\right], \left[x, t\right],\left[z, x\right]t^{-1}\right>\\
M\pi_1\left(S^4\setminus \nu(f)\right)/A_z&\cong\left<x, y, u| \left[x, u\right], \left[y, u\right],\left[x, y\right]u^{-1}\right>\text{.}
\end{align*}
as $M\pi_1\left(S^4\setminus\nu(f)\right)/A_i=M\pi_1\left(S^4\setminus\left(\nu(f_j)\sqcup\nu(f_k)\right)\right)$, where $i,j, k\in\{x,y,z\}$ and are pairwise distinct. We define
\begin{align*}
\Gamma_x&:=\left<y, z, s| \left[y, s\right], \left[z, s\right],\left[y, z\right]s^{-1}\right>\text{,}\\
\Gamma_y&:=\left<z, x, t| \left[z, t\right], \left[x, t\right],\left[z, x\right]t^{-1}\right>\text{,}\\
\Gamma_z&:=\left<x, y, u| \left[x, u\right], \left[y, u\right],\left[x, y\right]u^{-1}\right>\text{,}
\end{align*}
and $K:=\mathbb{Z}\Gamma_x\times\mathbb{Z}\Gamma_y\times\mathbb{Z}\Gamma_z$.

\begin{rem}
To clarify the above, we have taken a based link map and used the basing curves for each surface to define a meridian and then chosen isomorphisms to $F/F_3$ using the specified meridians.
\end{rem}

\begin{defn}
Let $f$ be a based three-component link map then
\[
\sigma^3(f)=\left(\sigma_x(f), \sigma_y(f), \sigma_z(f)\right):=\left(M\left(\lambda(f_x,f_x)\right),M\left(\lambda(f_y,f_y)\right),M\left(\lambda(f_z,f_z)\right)\right)\in K\text{.}
\]
\end{defn}
\begin{prop}\label{sigma}
The map $\sigma^3$ is a based link homotopy invariant.
\begin{proof}
By Proposition \ref{milncomp}, through a generic homotopy, the Milnor group of the fundamental group of the complement of two spheres, described by a link map, remains constant. We now check that under a generic homotopy the triple remains fixed.

Considering a link map $f$, we check that finger moves do not affect the intersection values. We first check that $\sigma_x(f)=M(\lambda\left(f_x, f_x\right))$ is not affected by any finger move on any component. The other components will follow similarly. It is clear that doing a finger move on the $x$ component does not affect the intersection number since $\lambda$ is a homotopy invariant. 

We now check that a finger move on the $y$ component does not affect $M\lambda(f_x, f_x)$. If we do a finger move on the $y$ sphere we change the fundamental group of the complement of $\pi_1(S^4\setminus \left(\nu(f_y)\coprod (f_z)\right))$ to $\pi_1(S^4\setminus \left(\nu(f_y)\coprod (f_z)\right))\big/\left<\left<\left[y, gyg^{-1}\right]\right>\right>$ for some $g\in \pi_1(S^4\setminus \left(\nu(f_y)\coprod (f_z)\right))$ \cite{Cas}. We thus have the following commutative diagram
\[
\begin{tikzcd}
\pi_1(S^4\setminus\left(\nu(f_y)\coprod (f_z)\right))\arrow[r]\arrow[d]& M\pi_1(S^4\setminus \left(\nu(f_y)\coprod (f_z)\right))\arrow[d, "\Id"]\\
\pi_1(S^4\setminus\left(\nu(f_y)\coprod (f_z)\right))\big/\left<\left<\left[y, gyg^{-1}\right]\right>\right>\arrow[r]&
M\left(\pi_1(S^4\setminus \left(\nu(f_y)\coprod (f_z)\right))\big/\left<\left<\left[y, gyg^{-1}\right]\right>\right>\right)\text{.}
\end{tikzcd}
\]
The horizontal maps are the Milnor quotient maps. Hence, group elements are not changed by  finger moves, since the finger move introduces relations already a consequence of the relations in $M\pi_1\left(S^4\setminus \left(f_y\sqcup f_z\right)\right)$. Hence, $M\lambda(f_x, f_x)$ is invariant under finger moves on the $y$ component. An analogous argument works for finger moves on the $z$ component. Hence $\sigma^3$ is invariant under finger moves.

Let $g$ be a generic immersed link map link homotopic to $f$. We can assume that that both maps have the same Euler number, since cusps do not change $\lambda$ nor the fundamental group of the complement. Let $F:\left(S^2_x\coprod S^2_y\coprod S^2_z\right)\times I\rightarrow S^4$ be a generic  homotopy such that $F(-, 0)=f$ and $F(-, 1)=g$. We may assume this is a regular homotopy since both Euler numbers are the same and by Proposition \ref{FQ} we may assume $F$ is a concatenations of finger moves, Whitney moves and isotopies. As finger moves are supported by an arc and Whitney moves are supported in the neighbourhood of a disc we may assume $F$ does all the finger moves when $t\in\left(0,\frac{1}{2}\right)$ and Whitney moves when $t\in\left(\frac{1}{2}, 1\right)$. By the earlier argument,  finger moves do not change the triple and thus the triple is equal on $f$ and $F(-, \frac{1}{2})$. Running the homotopy backwards from $g$ to $F(-,\frac{1}{2})$ is a sequence of finger moves, so the triple on $g$ and $F(-, \frac{1}{2})$ is the same. By transitivity, $\sigma^3$ is a homotopy invariant. 
\end{proof}
\end{prop}

\subsection{Link maps (unbased)}
Using the invariant for the based link maps we will develop an invariant for unbased link maps. This will come from taking a quotient of $K$ which corresponds to changing the choice of basing paths for each sphere.

We first analyse the changes to $\sigma^3$ which occur when we change the basing of a single component. After this, we describe what happens when all the basings change. This will give us an orbit space for the action on $K$ that correspond to changing the basing paths. This will be the correct quotient to describe an invariant of unbased link maps. 

\begin{lem}\label{changebase}
Let $f$ and $g$ be based link maps which are equal except that the basings of the $x$ component are different. Then for some $h\in \Gamma_x$ we have 
\[
M\lambda(g_x, g_x)=h\left(M\lambda(f_x, f_x)\right)h^{-1}\text{.}
\]
\end{lem}
\begin{proof}
Let $\gamma^f_x:I\rightarrow S^4$ and $\gamma^g_x:I\rightarrow S^4 $ be basing paths for the $x$ component. To prove the result we will work with the self-intersection number $\mu$ and use Theorem \ref{intersection} to calculate $\lambda$. Let $p\in f_x(S^2_x)$ be a self-intersection. Using the basing $\gamma_x^f$, the group element associated to the intersection is given by
\[
\gamma_x^f\cdot f(\delta_1)\cdot f(\overline{\delta}_2)\cdot \overline{\gamma}^f_x\text{.}
\]
Recall that $\delta_1$ and $\delta_2$ are paths on $S^2$ from $f^{-1}(\{\gamma_x^f(1)\})$ to the distinct elements $p_1, p_2\in f^{-1}_x(\{p\})$ respectively. Using the basing path $\gamma_x^g$, the group element associated to the self-intersection becomes 
\[
\gamma_x^g\cdot f(\delta_1^\prime)\cdot f(\overline{\delta}_2^\prime)\cdot \overline{\gamma}^g_x\text{,}
\]
where $\delta_1^\prime$ and $\delta_2^\prime$ are defined similarly to $\delta_1$, $\delta_2$. Let $\alpha:I\rightarrow \im(f_x)$ such that $\alpha(0)=\gamma_x^g(1)$ and $\alpha(1)=\gamma_x^f(1)$ such that $\alpha$ does not pass through an intersection at any point on the path. Also, define $h=\gamma_x^g\cdot \alpha\cdot\bar{\gamma}^f_x$. As $S^2$ is simply connected, $\alpha\cdot f(\delta_1)\simeq f(\delta_1^\prime)$ and $f(\bar{\delta}_2)\cdot\bar{\alpha}\simeq f(\bar{\delta}_2^\prime)$ . Using both of these homotopies we have
\begin{align*}
\gamma_x^g\cdot f(\delta_1^\prime)\cdot f(\bar{\delta}_2^\prime)\cdot \bar{\gamma}^g_x&\simeq\gamma_x^g\cdot\alpha\cdot f(\delta_1)\cdot f(\bar{\delta}_2)\cdot\bar{\alpha}\cdot\bar{\gamma}_x^g\\
&\simeq\gamma_x^g\cdot\alpha\cdot\bar{\gamma}_x^f\cdot \gamma_x^f\cdot f(\delta_1)\cdot f(\bar{\delta}_2)\cdot \bar{\gamma}_x^f\cdot \gamma_x^f\cdot\bar{\alpha}\cdot\bar{\gamma}_x^g\\
&=h\left(\gamma_x^f\cdot f(\delta_1)\cdot f(\bar{\delta}_2)\cdot \bar{\gamma}^f_x\right)\bar{h}\text{,}
\end{align*}
where $h \in\Gamma_x$. Using Theorem \ref{intersection} we have proven the result.
\end{proof}

\begin{figure}
\begin{tikzpicture}
\node [
    above right,
    inner sep=0] (image) at (0,0) {\includegraphics[scale=0.5]{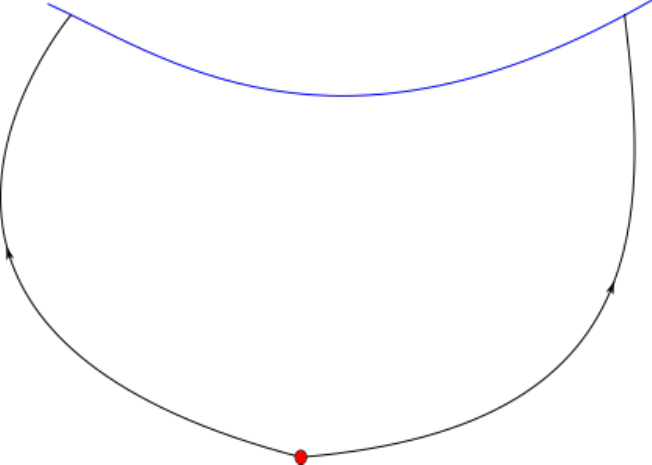}};
\node at (6,3){ $\gamma^f_x$};
\node at (-0.5,3){ $\gamma^g_x$};
\node at (2.7,3.45){$f_x$};
\end{tikzpicture}
     \caption{A schematic for the different choice of basing paths to the $x$ component, for the based link maps $f$ and $g$ in Lemma \ref{changebase}}
    \label{fig:Bpaths}
\end{figure}
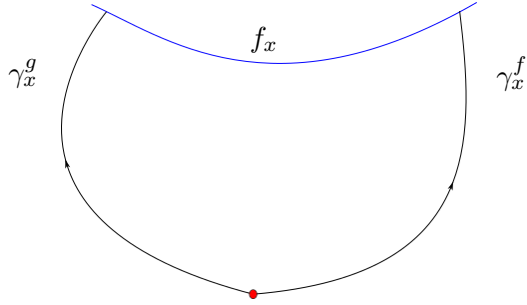

\begin{lem}\label{onebas}
Let $f$ be a based link map such the triple $\sigma^3(f)$ is equal to 
\[
\Bigg(\sum_{\left(i,j,k\right)\in\mathbb{Z}^3}a_{ijk}y^is^jz^k,\sum_{\left(i,j,k\right)\in\mathbb{Z}^3}b_{ijk}z^it^jx^k,\sum_{\left(i,j,k\right)\in\mathbb{Z}^3}c_{ijk}x^iu^jy^k\Bigg)\in K
\]
Then for a based link map $g$, which is equal to $f$ except that it differs by a basing path to the $x$ component, the triple $\sigma^3(g)$ is equal to \[
\Bigg(\sum_{\left(i,j,k\right)\in\mathbb{Z}^3}a_{ijk}y^is^{j-ak+bi}z^k,\sum_{\left(i,j,k\right)\in\mathbb{Z}^3}b_{ijk}z^it^{j+bk}x^k,\sum_{\left(i,j,k\right)\in\mathbb{Z}^3}c_{ijk}x^iu^{j-ai}y^k\Bigg)\in K\text{,}
\]
for some $a,b\in\mathbb{Z}$.
\begin{proof}
By the previous lemma we have $M\left(\lambda\left(g_x, g_x\right)\right)=hM\left(\lambda\left(f_x, f_x\right)\right)h^{-1}$, for some $h\in\Gamma_x$. Since $h=y^{v_1}s^{q_x}z^{w_1}$ for $v_1,q_x, w_1\in \mathbb{Z}$ then using the relations in $\Gamma_x$ we have
\begin{align*}
M\left(\lambda\left(g_x, g_x\right)\right)&=y^{v_1}s^{q_x}z^{w_1}\left(\sum_{\left(i,j,k\right)\in\mathbb{Z}^3}a_{ijk}y^is^jz^k\right)z^{-w_1}s^{-q_x}y^{-v_1}\\
&=\sum_{\left(i,j,k\right)\in\mathbb{Z}^3}a_{ijk}y^is^{j-{w_1}i+v_1k}z^k\text{.}
\end{align*}

We now show how the other components transform. The details of this are slightly more subtle, as the group elements associated to the intersections do not change when considered as elements of $M\pi_1(S^4\setminus f_x\sqcup f_z)$ and $M\pi_1(S^4\setminus f_x\sqcup f_y)$. The reason why the values in $\Gamma_y$ and $\Gamma_z$ change is because our choice of isomorphisms from $M\pi_1(S^4\setminus f_x\sqcup f_z)$  and $M\pi_1(S^4\setminus f_x\sqcup f_y)$, to $\Gamma_y$ and $\Gamma_z$ respectively, change when we change the basing path. In order to compute what our invariant looks like on $g$ we must describe the meridian of the first component of $f$ in terms of the specified meridians of $g$. 

Let $x^\prime$ be the meridian of the first component of $g$. Our goal is to find a loop $r\in \pi_1(S^4\setminus f)$ such $x=r^{-1}x^\prime r$, where we consider $r$ as in $\Gamma_y$ and $\Gamma_z$ and $r=h$ in $\Gamma_x$.

Let  $\gamma^f_{\varepsilon},\gamma^g_\varepsilon:I\rightarrow S^4$ be the basing paths of the $x$ components $f$ and $g$ respectively which stop at  $\gamma^f_x(1-\varepsilon)$ and $\gamma^f_x(1-\varepsilon)$. Recall that $\alpha: I\rightarrow f_x$ is a path between the two basing points of $g$ and $f$ starting at $\gamma_x^g(1)$ and ending at $\gamma^f_x(1)$ without going through any intersections. Let $\alpha^\prime$ be a normal push-off of $\alpha$ from the surface with starting point $\gamma^g_x(1-\varepsilon)$ and with end point $\gamma^f_x(1-\varepsilon)$. Let $p^g_x$ be the generator, based at $\gamma^g_x(1-\varepsilon)$, of $\pi_1(U\times D^2\setminus U\times\{0\})$ with positive linking number to $f_x$, where $U$ is an open subset around $\gamma_x^g(x)$ which trivialises the normal bundle. We have a homotopy $x\simeq\gamma^f_{\varepsilon}\cdot\overline{\alpha}^\prime\cdot p^g_x\cdot\alpha \cdot \overline{\gamma}^f_{\varepsilon}$, which we can see by comparing both pictures in Figure \ref{fig:example}. Hence,
\begin{align*}
x&\simeq\gamma^f_{\varepsilon}\cdot\overline{\alpha}^\prime\cdot p^g_x\cdot\alpha \cdot \overline{\gamma}^f_{\varepsilon}\\
&\simeq\gamma^f_{\varepsilon}\cdot\overline{\alpha}^\prime\cdot\overline{\gamma}^g_{\varepsilon}\cdot\gamma^g_{\varepsilon}\cdot p^g_x\cdot\overline{\gamma}^g_\varepsilon\cdot \gamma^g_{\varepsilon}\cdot\alpha \cdot \overline{\gamma}^f_{\varepsilon}\\
&=r^{-1}x^\prime r\text{,}
\end{align*}
where $r=\gamma^g_{\varepsilon}\cdot\alpha \cdot \overline{\gamma}^f_{\varepsilon}$ and, by definition, $x^\prime=\gamma^g_{\varepsilon}\cdot p^g_x\cdot\overline{\gamma}^g_\varepsilon$. Considering $r$ as an element inside $M\pi_1(S^4\setminus f_y\sqcup f_z)$ we can see that it is homotopic to $h\in\Gamma_x$. Using this we have 
\begin{align*}
M\left(\lambda(g_y, g_y)\right)&=\sum_{(i,j,k)\in\mathbb{Z}^3}b_{ijk}z^it^jx^k\\
&=\sum_{(i,j,k)\in\mathbb{Z}^3}b_{ijk}z^it^j(r^{-1}x^\prime r)^k,
\end{align*}
where we consider $r\in \Gamma_y$. Similarly we have 
\[
M\left(\lambda(g_z, g_z)\right)=\sum_{(i,j,k)\in\mathbb{Z}^3}c_{ijk}(r^{-1}xr)^iu^jy^k, 
\]
where we consider $r\in \Gamma_z$. We consider $r$ inside the different quotients, $\Gamma_x$, $\Gamma_y$ and $\Gamma_z$. Inside these quotients we find that $r$ is equal to $y^{v_1} s^{q_x}z^{w_1}\in\Gamma_x$, $z^{v_2} t^{q_y}x^{w_2}\in\Gamma_y$ and $x^{v_3}u^{q_z}y^{w_3}$, where $v_i, w_i\in\mathbb{Z}$ for $1\leq i\leq 3$.
 \begin{cla}\label{link com}
In the above we have $v_1=w_3$, $w_1=v_2$ and $w_2=v_3$.
 \end{cla}
 From this claim we set $v_1=-a$ and $w_1=-b$. We now have that the triple for $g$, substituting in the various quotients of $r$, and using $\Gamma_i$ relations we have
 \begin{align*}
\bigg(&\sum_{\left(i,j, k\right)\in\mathbb{Z}^3}a_{ijk}y^{-a}z^{-b}y^is^jz^kz^by^a,\sum_{\left(i,j, k\right)\in\mathbb{Z}^3}b_{ijk}z^it^j(z^bx^\prime z^{-b})^k,\sum_{\left(i,j, k\right)\in\mathbb{Z}^3}c_{ijk}(y^ax^\prime y^{-a})^iu^jy^k\bigg)\\
&=\bigg(\sum_{\left(i,j, k\right)\in\mathbb{Z}^3}a_{ijk}y^is^{j-ak+bi}z^k,\sum_{\left(i,j, k\right)\in\mathbb{Z}^3}b_{ijk}z^it^{j+bk}(x^\prime)^k ,\sum_{\left(i,j, k\right)\in\mathbb{Z}^3}c_{ijk}(x^\prime)^iu^{j-ai}y^k\bigg)\text{.}
\end{align*}
We now replace $x^\prime$ with $x$ to identify the $x$ meridian given by $g$ to get
\[
\bigg(\sum_{\left(i,j, k\right)\in\mathbb{Z}^3}a_{ijk}y^is^{j-ak+bi}z^k,\sum_{\left(i,j, k\right)\in\mathbb{Z}^3}b_{ijk}z^it^{j+bk}x^k ,\sum_{\left(i,j, k\right)\in\mathbb{Z}^3}c_{ijk}x^iu^{j-ai}y^k\bigg)
\]
To complete the proof, we need to account for the different possible choice of $r$ resulting from the choice of push off of $\alpha$. The different choice of $\alpha$ does not effect $a$ and $b$ as these are linking numbers of the path with the $y$ and $z$ component and upon the inclusion of the $x$ component all the different paths become homotopic and thus these values are unchanged by our choice. The only value which is changed is the power of $x$ in $\Gamma_y$ and $\Gamma_z$ but this does not affect the end results as the triple is independent of $p^\prime$. This completes the proof modulo the proof of Claim \ref{link com}. 
\end{proof}
\end{lem}

\begin{rem}
It is clear from the proof of Lemma \ref{onebas} that we only need to consider the loop $r$ as an element of $\Gamma_x$ to know how to change the triple for the new basing path and, since elements in the center do not matter, we only need to look at the curve as an element of $\Gamma_x/\left[\gamma_x, \gamma_x\right]\cong H_1(S^4\setminus f_y\sqcup f_z)$. As $v_1$ is the linking of the loop with the $S_y^2$ sphere and $w_1$ is the linking number of the loop $S^2_z$, it is clear we only need to know the linking number of the loop with the other components to calculate how the triple transforms.
\end{rem}

\begin{figure}%
    \centering
    \subfloat{{\begin{tikzpicture}
\node [
    above right,
    inner sep=0] (image) at (0,0) {\includegraphics[scale=0.5]{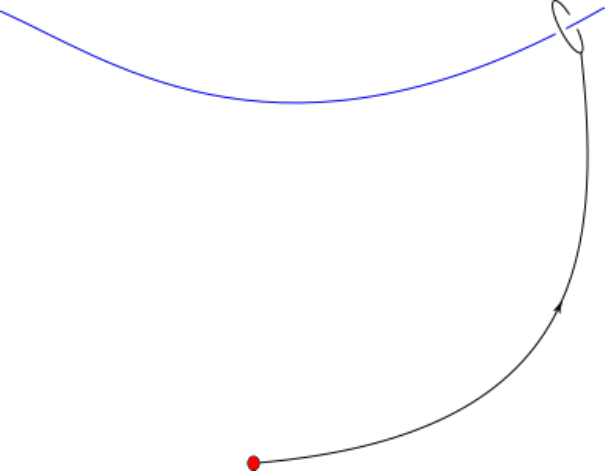}};
\node at (4.6,2){$x$};
\node at (2.7,3.45){$f_x$};
\end{tikzpicture}}}%
    \qquad\quad\quad\quad
    \subfloat{{\begin{tikzpicture}
\node [
    above right,
    inner sep=0] (image) at (0,0) {\includegraphics[scale=0.5]{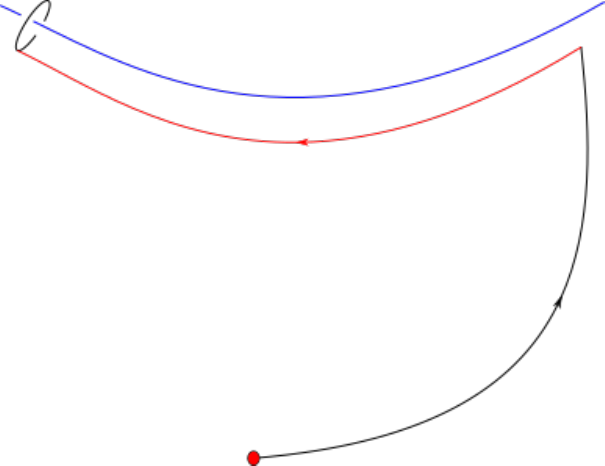}};
\node at (5.3,3){ $\gamma^f_{\epsilon}$};
\node at (2.7,3.45){$f_x$};
\node at (0.9, 4){$p^g_{x}$}; 
\node at (2.7,2.5){$\alpha$};
\end{tikzpicture} }}%
    \caption{On the left is a schematic showing the meridian $x$ associated to the surfaces basing path. On the right is a  schematic showing the result of a homotopy of the $x$ meridian, which uses the push-off of the surface $\alpha$.}%
    \label{fig:example}%
\end{figure}
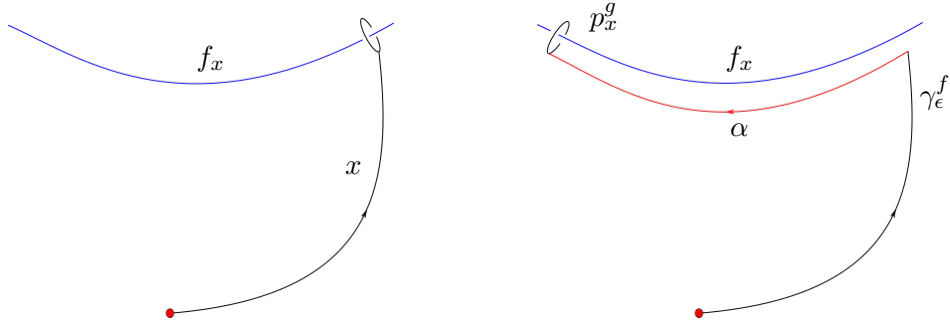

We now prove the claim that was in the previous proof.
 \begin{proof}[Proof of Claim \ref{link com}]
We will prove that the following diagram commutes
\[
\begin{tikzcd}
M\pi_1(S^4\setminus f)\arrow[d,"q"]\arrow[rr, "\Ab"]&&\mathbb{Z}^3\arrow[d]\\
\Gamma_x\times\Gamma_y\times\Gamma_z\arrow[rr,"\Ab\times \Ab\times \Ab"]&&\mathbb{Z}^2\times\mathbb{Z}^2\times\mathbb{Z}^2\text{,}
\end{tikzcd}
\]
where $\Ab$ is the quotient map to the abelianisation for the various groups involved. The map $q$ takes a group and sends it to its various representatives in each factor. Notice that $\left(\Ab\times \Ab\times \Ab\right)(\Gamma_x\times \Gamma_y\times \Gamma_z)=H_1(S^4\setminus f_y\sqcup f_z)\times H_1(S^4\setminus f_z\sqcup f_x)\times H_1(S^4\setminus f_x\sqcup f_y)\cong\mathbb{Z}^2\times\mathbb{Z}^2\times\mathbb{Z}^2$ and $H_1(S^4\setminus f)\cong \mathbb{Z}^3$. 
The map $\mathbb{Z}^3\rightarrow \mathbb{Z}^2\times\mathbb{Z}^2\times\mathbb{Z}^2$ is given by
\[
x^\alpha y^\beta z^\gamma\mapsto \left((\beta,\gamma), (\alpha, \gamma), (\alpha, \beta)\right)
\]
Proving the above diagram commutes is equivalent to showing that
\[
\begin{tikzcd}
M\pi_1(S^4\setminus f)\arrow["\Ab"]{d}\arrow["i"]{r}&\Gamma_x\arrow["\Ab"]{d}\\
H_1(S^4\setminus f)\arrow["H(i)"]{r}&H_1(S^4\setminus f_y\sqcup f_z)
\end{tikzcd}
\]
commutes, as the roles of $x$, $y$ and $z$ are symmetric. We can write $r=x^\alpha y^\beta z^\gamma\eta$ where $\eta$ is a product of commutators. Then $\Ab(i(r))=\Ab(y^\beta z^\gamma s^p)=y^\beta z^\gamma$. Computing the other path, we have $H(i)\left(\Ab(r)\right)=H(i)(x^\alpha y^\beta z^\gamma)=y^\beta z^\gamma$. This proves the claim.
 \end{proof}

We now show how the triple changes when all three components change their basing paths 
\begin{cor}
Let $g_1$ be a based link map with triple $\sigma^3(g_1)$ equal to
\[
\Bigg(\sum_{\left(i,j,k\right)\in\mathbb{Z}^3}a_{ijk}y^is^jz^k,\sum_{\left(i,j,k\right)\in\mathbb{Z}^3}b_{ijk}z^it^jx^k,\sum_{\left(i,j,k\right)\in\mathbb{Z}^3}c_{ijk}x^is^jy^k\Bigg)\in K\text{.}
\]
Let $g_2$ be the same link map as $g_1$ with  potentially different basings then
\begin{multline*}
\sigma^3(g_2)=\Bigg(\sum_{\left(i,j,k\right)\in\mathbb{Z}^3}a_{ijk}y^is^{j+(b-c)i+(f-a)k}z^k,\sum_{\left(i,j,k\right)\in\mathbb{Z}^3}b_{ijk}z^it^{j+(d-e)i+(b-c)k}x^k,\\\sum_{\left(i,j,k\right)\in\mathbb{Z}^3}c_{ijk}x^is^{j+(f-a)i+(d-e)k}y^k\Bigg)\in K\text{,}
\end{multline*}
for some $a,b,c,d,e,f\in\mathbb{Z}$.

\begin{proof}
Let $r_x=y^az^bs^\alpha\in\Gamma_x$, $r_y=z^cx^dt^\beta\in\Gamma_y$ and $r_z=x^ey^du^\gamma\in\Gamma_z$ be the loops we construct from the two basings on each component. We now repeat the proof of Lemma \ref{onebas} on each sphere: changing the $x$-basing we have
\[
\bigg(\sum_{\left(i,j, k\right)\in\mathbb{Z}^3}a_{ijk}y^is^{j-ak+bi}z^k,\sum_{\left(i,j, k\right)\in\mathbb{Z}^3}b_{ijk}z^it^{j+bk}x^k ,\sum_{\left(i,j, k\right)\in\mathbb{Z}^3}c_{ijk}x^iu^{j-ai}y^k\bigg);
\]
and changing the $y$-basing
\[
\bigg(\sum_{\left(i,j, k\right)\in\mathbb{Z}^3}a_{ijk}y^is^{j-ak+(b-c)i}z^k,\sum_{\left(i,j, k\right)\in\mathbb{Z}^3}b_{ijk}z^it^{j+(b-c)k+di}x^k ,\sum_{\left(i,j, k\right)\in\mathbb{Z}^3}c_{ijk}x^iu^{j-ai+dk}y^k\bigg).
\]
Changing the $z$-basing we arrive at the desired result,
\[
\bigg(\sum_{\left(i,j, k\right)\in\mathbb{Z}^3}a_{ijk}y^is^{j+(b-c)i+(f-a)k}z^k,\sum_{\left(i,j, k\right)\in\mathbb{Z}^3}b_{ijk}z^it^{j+(b-c)k+(d-e)i}x^k,\sum_{\left(i,j, k\right)\in\mathbb{Z}^3}c_{ijk}x^iu^{j+(f-a)i+(d-e)k}y^k\bigg)\text{.}
\]
One can check that the above does not depend on the order in which we changed the basing paths. However, it is clear geometrically that this is the case.
\end{proof}
\end{cor}

We wish to show that we can achieve all values of $a,b,c,d,e,f\in\mathbb{Z}$ by changing basing paths. Let $g$ be a link map and define $B_g$ to be the set consisting of triples of basing paths to the $x$, $y$ and $z$ spheres. We define a map $\zeta:B_g\times B_g\rightarrow \mathbb{Z}^6$. Let $Q,Q^\prime\in B_g$. Let $\gamma_x$, $\gamma_y$ and $\gamma_z$ be the basing paths for each sphere given by $Q$ and let $\gamma^\prime_x$, $\gamma^\prime_y$ and $\gamma^\prime_z$ be the basing paths in $Q^\prime$. Define a loop
\[
h_x:=\gamma_x\cdot\delta\cdot\bar{\gamma}_x^\prime,
\]
where $\delta$ is the unique path up to homotopy between the points $\gamma_x(1)$ and $\gamma^\prime_x(1)$ which is the image of a path in the pre-images. We define $h_y$ and $h_z$ similarly. Thus we define 
\[
\zeta\left(Q, Q^\prime\right):=\big(\lk\left(h_x, g(S^2_y)\right), \lk\left(h_x, g(S^2_z)\right),\lk\left(h_y, g(S^2_z)\right), \lk\left(h_y, g(S^2_x)\right),\lk\left(h_z, g(S^2_x)\right),\lk\left(h_z, g(S^2_y)\right)\big)\text{.}
\]

\begin{lem}
The map, $\zeta:B_g\times B_g\rightarrow \mathbb{Z}^6$, is surjective.
\begin{proof}
Let $m_x$, $m_y$ and $m_z$ be meridians defined using the basing paths of $Q$ and let $a,b,c,d,e,f\in\mathbb{Z}$. Let $Q_m$ be the basing paths consisting of the paths $m_y^a\cdot m_z^b\cdot\gamma_x$, $m_z^c\cdot m_x^d\cdot\gamma_y$ and $m_x^e\cdot m_y^f\cdot \gamma_z$. Hence, $\zeta\left(Q_m, Q\right)=\left(a,b,c,d,e,f\right)$\text{.}

\end{proof}
\end{lem}

\begin{defn}\label{action}
Let $\left(a,b,c,d,e,f\right)\in\mathbb{Z}^6$ and
\[
\bigg(\sum_{\left(i,j, k\right)\in\mathbb{Z}^3}a_{ijk}y^is^jz^k,\sum_{\left(i,j, k\right)\in\mathbb{Z}^3}b_{ijk}z^it^jx^k,\sum_{\left(i,j, k\right)\in\mathbb{Z}^3}c_{ijk}x^iu^jy^k\bigg)\in K\text{.}
\]
Then we define
\begin{align*}
&\left(a,b,c,d,e,f\right)\cdot\bigg(\sum_{\left(i,j, k\right)\in\mathbb{Z}^3}a_{ijk}y^is^jz^k,\sum_{\left(i,j, k\right)\in\mathbb{Z}^3}b_{ijk}z^it^jx^k,\sum_{\left(i,j, k\right)\in\mathbb{Z}^3}c_{ijk}x^iu^jy^k\bigg)\\
&=\bigg(\sum_{\left(i,j,k\right)\in\mathbb{Z}^3}a_{ijk}y^is^{j-ci+fk-ak+bi}z^k,\sum_{\left(i,j, k\right)\in\mathbb{Z}^3}b_{ijk}z^it^{j-ei+bk-ck+di}x^k,\sum_{\left(i,j, k\right)\in\mathbb{Z}^3}c_{ijk}x^iu^{j-ai+ck-ek+fi}y^k\bigg)\text{.}
\end{align*}
\end{defn}

\begin{lem}\label{act}
The operation described in Definition \ref{action} is a group action on $K$ by $\mathbb{Z}^6$.
\begin{proof}
 One must check that for all $g\in K$ $(0,0,0,0,0,0)\cdot g=g$ and 
\[
(a,b,c,d,e,f)\cdot\left((a^\prime,b^\prime,c^\prime,d^\prime,e^\prime, f^\prime)\cdot g\right)=(a+a^\prime,b+b^\prime,c+c^\prime,d+d^\prime,e+e^\prime,f+f^\prime)\cdot g\text{.}
\]
This is clear from the algebra.
\end{proof}
\end{lem}

Let $\widetilde{K}$ be the orbit space of this action. 

\begin{defn}
Let $f$ be a three-component link map. Then we define a map $\widetilde{\sigma}^3:\LM_{2,2,2}^4\rightarrow\widetilde{K}$ by
\[
\widetilde{\sigma}^3\left(f\right):=[\sigma^3(f)]\in \widetilde{K}\text{.}
\]
\end{defn}
\begin{prop}
The map $\widetilde{\sigma}^3$ is a well defined map on $\LM_{2,2,2}^4$.
\begin{proof}
We showed that in the based case under link homotopy the triple remains fixed. It is clear from Lemma \ref{onebas} and the definition of $\widetilde{K}$ that the map is well defined.
\end{proof}
\end{prop}

The invariant $\widetilde{\sigma}^3$ contains all the data of the Kirk invariant on its two-component sublink maps.
\begin{lem}\label{containment}
Let $i:\LM_{2,2,2}^4\rightarrow \LM_{2,2}^4$ be the map which forgets about the $i$th sphere. Additionally, let $p_i:K\rightarrow \left(\mathbb{Z}\left[\mathbb{Z}\right]\right)^2$ be the map given by projecting onto the factors which are not $i$ and setting the $i$th meridian equal to $1$. Then the following diagram commutes
\begin{equation}
\begin{tikzcd}
\LM_{2,2,2}^4\arrow[d, "\widetilde{\sigma}^3"]\arrow[r, "i"]& \LM_{2,2}^4\arrow[d, "\sigma"]\\
K\arrow[r, "p_i"]& \left(\mathbb{Z}\left[\mathbb{Z}\right]\right)^2,
\end{tikzcd}
\end{equation}

where $i$ is the map that forgets the $i$th component, and $p_i$ is the map where you project onto two factors which are not $i$ and sets $i$th component's meridian equal to $1$.
\end{lem}

\subsection{Examples of three-component linking behaviour}

 Our three-component invariant, $\tilde{\sigma}^3$, can differentiate between the trivial link map and link maps which are trivial when any component is removed, and thus can detect linking behaviour which only occurs when there are at least three-components.

\begin{theo}\label{R1}
Let $L$ be the link given in the centre of Figure \ref{fig:Pic1}. Then we can apply a null homotopy of components of the link to construct a representative of $f\in \LM_{2,2,2}^4$ with some choice of basing path such that
\[
\tilde{\sigma}^3(f)=\left(z\left(s-1\right)+z^{-1}\left(s^{-1}-1\right),\, 0,\, x\left(1-u\right)+x^{-1}\left(1-u^{-1}\right)\right)\neq 0.\]
Hence, $F$ is link homotopically non-trivial. However, the removal of any component gives a trivial component two-component link map.
\begin{figure}
    \centering
    \includegraphics[scale=1]{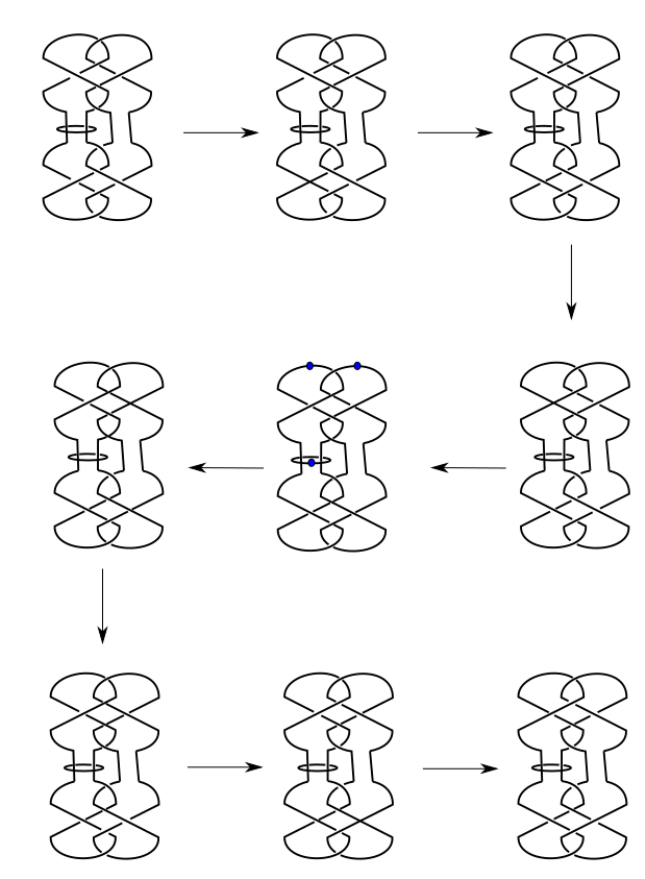}
    \caption[Example of non-trivial three-component link map]{This is a sequence of times slices in $D^3\times I$ inside $S^4$ describing a link map where we have specified a basing paths and meridians on the central time slice by blue points on each link component. The $x$ component is the component at the top on the left, the second components is the circle round the band and the third component is the right most component. The orientation of each components travels through the basepoint to the left.} 
    \label{fig:Pic1}
\end{figure}
\begin{proof}
Using Figure \ref{fig:Pic1}, we have a description of a three-component link map inside $S^4$ which we call $f$.  To compute the first component of $\sigma^3$, consider the null-homotopy from the central time slice to the end. The sign of the self-intersections in the corresponding link map is labelled on Figure \ref{fig:Pic1}. The corresponding element of $\gamma_x$ at the first intersection on this region is $z$ (we could have chosen $\overline{z}$ here). The corresponding element of $\Gamma_x$ of the next self-intersection in this region is given by 
$zyz^{-1}y^{-1}z^{-1}\in \Gamma_x$. Hence, the first component of $\sigma^3(f)$ is given by
\[
z\left(\left[y,\,z\right]-1\right)+z^{-1}\left(\left[z,\,y\right]-1\right)\in\mathbb{Z}\Gamma_x\text{.}
\]
We now compute the third component of $\sigma^3$. Using the Wirtinger presentation one can show that the meridian of the first component, below the second component, is given by $y^{-1}xy$. We now study the homotopy from the start to the central time slice. The group element associated to the first double point in this region is $x\in \Gamma_z$. The group element associated to the next intersection in this region is $y^{-1}xy\in \gamma_z$ this is equivalent to the group element $x\left[x,\, y\right]$. Hence the third component is 
\[
x\left(1-\left[x,y\right]\right)+x^{-1}\left(1-\left[y,\,x\right]\right)\in\mathbb{Z}\Gamma_z\text{.}
\]
The second component is embedded and thus the second component of $\sigma^3$ vanishes.

If we remove a sphere the link becomes trivial. To see this, we use repeated applications of Lemma \ref{containment}. Notice that $p_i(\widetilde{\sigma}^3(f))=0$ for each $i$. Using that the Kirk invariant is injective \cite{ST} this implies that $i(f)$ is a trivial link map for all $i$ which proves the result. 
\end{proof}
\end{theo}

Paul Kirk had a version of his invariant for more than two components \cite{PK}. If we consider the link map constructed using Figure \ref{fig:Pic1}, his invariant fails to be able to distinguish between this non-trivial link map and a trivial link map. We will give a name for link maps which a trivial once a component is removed.
 
\begin{defn}
 We call a link map \emph{Brunnian} if the removal of any single component results in a trivial link map up to link homotopy. We will denote the set of three-component Brunnian link maps by $\BL_{2,2,2}^4$.
\end{defn}

\begin{exam}
Let $f$ be the link map from Theorem \ref{R1} with the same basing paths and orientations. Let $g=f\circ k$ where $k:S^2_x\coprod S^2_y\coprod S^2_z\rightarrow S^2_x\coprod S^2_y\coprod S^2_z$ which is a reflection on $S^2_y$ and the identity elsewhere. This gives us
\begin{align*}
    \sigma^3(f)&=\left(z\left(s-1\right)+z^{-1}\left(s^{-1}-1\right),\, 0,\, x\left(1-u\right)+x^{-1}\left(1-u^{-1}\right)\right)\\
    \sigma^3(g)&=\left(z\left(s^{-1}-1\right)+z^{-1}\left(s-1\right),\, 0,\, x\left(1-u^{-1}\right)+x^{-1}\left(1-u\right)\right)\text{.}
\end{align*}
If we consider the based verison of the invariant it is clear that these are two different based link maps. However, we must check that the unbased case gives two distinct equivalence classes. Let us fix the representative of $\sigma^3(g)$ as above. The equivalence class for $\sigma^3(f)$ is described by
\[
\left(z\left(s^{1-a}-s^{-a}\right)+z^{-1}\left(s^{-1+a}-s^{a}\right),0, x\left(u^{-a}-u^{1-a}\right)+x^{-1}\left(u^{a}-u^{-1+a}\right)\right)\text{,}
\]
where $a\in\mathbb{Z}$. Taking the difference and looking at the first component we get
\[
z(s^{1-a}-s^{-a}-s^{-1}+1)+z^{-1}(s^{-1+a}-s^{a}-s+1).
\]
A necessary condition for both link maps to be equal is to have an $a\in\mathbb{Z}$ such that 
\[
s^{1-a}-s^{-a}-s^{-1}+1=0\text{.}
\]
Clearly, such an $a$ must have $|a|<2$. Checking each value of $a$ remaining, it is evident we can never solve the above equation. Hence, as unbased link maps they are not equal.
\end{exam}
This establishes the following.
\begin{theo}
For each $n\geq3$ there exists link maps $f=f_1\sqcup\cdots\sqcup f_n$ and $f^\prime=f_1^\prime\sqcup\cdots\sqcup f_n^\prime$ such that for $i$, $f_i(S^2)=f_i^\prime(S^2)$, but $f$ and $f^\prime$ are not link homotopic
\end{theo}
\subsection{New invariants}\label{newinv}
 Frequently, it is difficult to differentiate two elements of $\LM_{2,2,2}^4$ using $\tilde{\sigma}^3$ since it can be difficult to tell if two representatives of elements of $\widetilde{K}$ are in the same equivalence class.  This subsection will focus on extracting new invariants from $\tilde{\sigma}^3$ which also are independent of our choice of basing path for each component.
 
 Let $f\in \LM_{2,2,2}^4$ and write
 \[
 \tilde{\sigma}^3\left(f\right)=\left(\sum_{\left(i,j,k\right)\in\mathbb{Z}^3}a^f_{ijk}y^is^jz^k, \sum_{\left(i,j,k\right)\in\mathbb{Z}^3}b^f_{ijk}z^it^jx^k,\sum_{\left(i,j,k\right)\in\mathbb{Z}^3}c^f_{ijk}x^iu^jy^k\right)
 \]
 for some choice of basing paths for each sphere, where $a_{ijk}^f, b_{ijk}^f, c_{ijk}^f\in\mathbb{Z}$.
 Let, $A_f$ be the unordered $n$-tuple of non-zero $a^f_{ijk}$. Define $B_f$ and $C_f$ similarly. We have the following
 \begin{prop}
Let $f, g$ be link homotopic link maps. Then $A_f=A_g$, $B_f=B_g$ and $C_f=C_g$.

 \begin{proof}
 The proposition follows from immediately from the definition of $\widetilde{\sigma}^3$.
 \end{proof}
 \end{prop}
 \begin{prop}
Let $f\in \LM_{2,2,2}^4$ with some choice of basing path. Suppose $g$ is a link map link homotopic to $f$. Then for any choice of basing path there exist an $l\in \mathbb{Z}$ such that $a_{ilk}^g=a_{ijk}$ and 
\[
j\equiv l\mod \gcd\left(i, k\right).
\]
\begin{proof}
It is clear that the exponent of $s$ does not change if we do a link homotopy of a based link map, so we must show that this is independent of our choice of basing path. Notice that a change of basing paths does the following transformation: 
\[
y^is^jz^k\mapsto y^is^{j+(b-c)i+(f-a)k}z^k
\]
Hence, taking $j$ modulo $\gcd\left(i, k\right)$ we have a number independent of our choice of basing path, which proves the proposition.
\end{proof}
\end{prop}
 Similar invariants can be derived by looking at the power of the $t$ and $u$ terms in $\sigma_y$ and $\sigma_z$ respectively. These invariants are similar to Milnor's triple linking number introduced in \cite{LGM}, with the only significant difference being that there is no obvious sense in which we have symmetry relations. However, this invariant requires us to throw away information of the other intersections. We will now construct a similar invariant but which keeps track of all changes to the intersections.
 
Let $n_1$, $n_2$ and $n_3$ non-negative integers with $n=2n_1+2n_2+2n_3$. Let $\mathcal{S}_{n_k}$ be the symmetric group on $n_k$ variables when $n_k$ is non-zero. We denote $\mathcal{S}_{n_1}\times \mathcal{S}_{n_2}\times \mathcal{S}_{n_3}$ by $\mathcal{S}_{(n_1,n_2,n_3)}$. This group has an action on $\mathbb{Z}^n=\left(\mathbb{Z}^{2}\right)^{n_1}\times\left(\mathbb{Z}^{2}\right)^{n_2}\times\left(\mathbb{Z}^{2}\right)^{n_3}$ given by permuting components.

We can think of the action of $\mathcal{S}_{(n_1, n_2, n_3)}$ as being represented by a group of matrices $P_{(n_1, n_2, n_3)}$.
 
 We define
\[
A_{k}\left(\mathbb{Z}^n\right):=\{x+N\mid x\in\mathbb{Z}^n\text{ and }N\text{ is a submodule of }\mathbb{Z}^n\text{ of rank }k\}\text{.}
\]
\begin{defn}
Take the quotient space
\[
\mathcal{A}_{\left(n_1,n_2,n_3\right)}:=\coprod_{k=0}^3A_k\left(\mathbb{Z}^n\right)/P_{\left(n_1,n_2,n_3\right)}\text{,}
\]
and define
\[
\mathcal{A}:=\coprod_{\left(n_1,n_2,n_3\right)\in\mathbb{Z}_{\geq0}^3}\mathcal{A}_{\left(n_1,n_2,n_3\right)}\text{.}
\]
\end{defn}
 We now wish to construct a map 
 \[
 \mu:K\rightarrow \mathcal{A}
 \]
which descends to a well-defined map 
\[
 \overline{\mu}:\widetilde{K}\rightarrow \mathcal{A}.
\]
Let
 \[
 v=\left(\sum_{\left(i,j, k\right)\in\mathbb{Z}^3}a_{ijk}y^is^jz^k,\sum_{\left(i^\prime,j^\prime, k^\prime\right)\in\mathbb{Z}^3}b_{i^\prime j^\prime k^\prime}z^{i^\prime}t^{j^\prime}x^{k^\prime},\sum_{\left(i^{\prime\prime},j^{\prime\prime}, k^{\prime\prime}\right)\in\mathbb{Z}^3}c_{i^{\prime\prime}j^{\prime\prime}k^{\prime\prime}}x^{i^{\prime\prime}}u^{j^{\prime\prime}}y^{k^{\prime\prime}}\right)
 \]
 
and $n_{1}=|A_v|$, $n_{2}=|B_v|$ and $n_{3}=|C_v|$. Additionally, place an ordering on the elements of $A_v$, $B_v$, and $C_v$. We write the $l$th element of $A_v$, where $1\leq l\leq n_{1}$, as $a_{i_lj_lk_l}$ corresponding to the group element $y^{i_l}s^{j_l}z^{k_l}$. We write the $l$th element of $B_v$ and $l$th element of $C_v$ similarly. We now define $\mu(v)$ to be the set of vectors that satisfy 
\begin{align*}
\sum_{l=1}^{n_1}(j_l+q_1k_l+q_2i_l)e_{2l-1}+&\sum_{l=1}^{n_1}(a_{i_lj_lk_l})e_{2l}+\sum_{l=1}^{n_2}\left(j^{\prime}_l+q_2k_l^\prime+q_3i^\prime_l\right)e_{n_1+2l-1}+\sum_{l=1}^{n_2}(b_{i^\prime_lj^\prime_lk^\prime_l})e_{2l}\\&+\sum_{l=1}^{n_3}\left(j^{\prime\prime}_l+q_1i^{\prime\prime}_l+q_3k^{\prime\prime}_l\right)e_{n_1+n_2+2l-1}+\sum_{l=1}^{n_3}(c_{i^{\prime\prime}_lj^{\prime\prime}_lk^{\prime\prime}_l})e_{2l}\in\mathcal{A}\text{.}
\end{align*}

where $q_1,q_2,q_3\in \mathbb{Z}$ and $e_l$ to the vector which has all zeros except for a $1$ in the $l$th position.
\begin{prop}
The map $\mu:K\rightarrow \mathcal{A}$ is well defined and descends to a map on the quotient $\overline{\mu}:\widetilde{K}\rightarrow \mathcal{A}$.
\begin{proof}
Indeterminacy arises from our choice of ordering on the intersections. However, this is accounted for by the choice of quotients used to construct $\mathcal{A}$. 

The submodule associated to the affine space defined by $\mu$ is dependent on the powers of the non-commutator generators of each group element which is unchanged by the action of $\mathbb{Z}^6$, which defines $\widetilde{K}$. Furthermore, the vector given by the powers of the commutator terms all lie in the same affine space so when considering different representatives we assign the same affine space.
\end{proof}
\end{prop}

This invariant is similar to the total Milnor quotient in \cite{MP} by Davis, Nagel, Powell and Orson. They made use of exterior algebra since Milnor's $\bar{\mu}$ invariants obey some relations by permuting elements. We do not have these symmetry relations, as our objects that are ``triple linking'' are a closed loop and two immersed two-spheres.

The map $\overline{\mu}$ is clearly stronger than our previous invariant which considered the power of the commutator term of the intersection which is shown in the following example.

\begin{exam}\label{examp}
Let $f$ be the link map from Theorem \ref{R1} and $g$ be a similar link map which can be represented by a similar movie but the band travels around the second component twice. We thus have
\[
\widetilde{\sigma}^3(f)=\left(z(s-1)+z^{-1}(s^{-1}-1),\, 0,\,x(1-u)+x^{-1}(1-u^{-1})\right)
\]
and
\[
\widetilde{\sigma}^3(g)=\left(z(s^2-1)+z^{-1}(s^{-2}-1),\, 0,\,x(1-u^2)+x^{-1}(1-u^{-2})\right).
\]
Computing $\overline{\mu}$ for each we get
\[
\overline{\mu}\left(\sigma^3\left(f\right)\right)=\Big(
1+q_1,\,1,\,q_1,\,-1,\,-1-q_1,\,1,\,-q_1,\,-1, 1+q_1,\,-1,\,q_1,\,1,\,-1-q_1,\,-1,\,-q_1,\,1\Big)
\]
and
\[
\overline{\mu}\left(\sigma^3\left(f\right)\right)=\Big(
2+q_1,\,1,\,q_1,\,-1,\,-2-q_1,\,1,\,-q_1,\,-1, 2+q_1,\,-1,\,q_1,\,1,\,-2-q_1,\,-1,\,-q_1,\,1\Big)\text{,}
\]
for some choice of ordering of the intersections. 

Each of these affine lines has a unique real affine line in $\mathbb{R}^8$ which intersects every point in the affine subspace. Computing the distance of these ``completed'' affine lines from the origin we get $10$ and $16$ respectively. The actions of $P_{4,0,4}$ are orthogonal so both distances are unchanged by the action of $P_{4,0,4}$.
\end{exam}

If we post compose $\sigma^3$ with $\overline{\mu}$ we have a map from $\LM_{2,2,2}^4\rightarrow \mathcal{A}$, we will often refer to this map $\overline{\mu}$ as it will be clear from context which domain we are considering.

\begin{prop}
Let $f\in \LM_{2,2,2}^4$ such that the representative affine space $\overline{\mu}\left(f\right)$ is a single vector. Then $f\in \BL_{2,2,2}^4$.
\begin{proof}
If $\overline{\mu}(f)$ is represented by a single vector then, in $\sigma^3(f)$, each component is given by a Laurent polynomial in the commutator terms, $s$, $t$ and $u$. Remove a component to get a two-component link map. Using Proposition \ref{containment} three times to compute the Kirk invariant and using the injectivity of the Kirk invariant proves the result.
\end{proof}
\end{prop}

For our final invariant on $\widetilde{K}$ we seek to place a notion of size of elements of our invariant.
\begin{defn}
 Let $A\subset\mathbb{Z}^3$ be finite and $d:\mathbb{Z}^3\times\mathbb{Z}^3\rightarrow \mathbb{R}$ be a metric. We call 
 \[
 D_d\left(A\right):=\max_{\left(a,b\right)\in A\times A}d(a,b),
 \]
 the \emph{diameter} of $A$.
\end{defn}

Take a representative of $\sigma^3\left(f\right)$, we define the following subset of $\mathbb{Z}^3$
\[
A^\prime_f:=\left\{\left(i, j, k\right)\in \mathbb{Z}^3| a_{ijk}\neq0\right\},
\]
We can define the sets $B^\prime_f$ and $C^\prime_f$ analogously. The sets $A_f^\prime$ and $B_f^\prime$ and $C_f^\prime$ are contained within equivalence classes of subsets of $\mathbb{Z}^3$, which we will denote by $\left[A_f^\prime\right]$, $\left[B_f^\prime\right]$ and $\left[C_f^\prime\right]$, where the equivalences classes are given by the action of $\mathbb{Z}^6$ on $\mathbb{Z}\Gamma_i$.

%We define the notion of a Width for these subsets of $\mathbb{Z}^3$.
\begin{defn}
Let $f\in K$ and $d:\mathbb{Z}^3\times\mathbb{Z}^3\rightarrow \mathbb{R}$ be a metric. We define
\[
W\left(f\right):=\min_{f\in[f]}\left(D_d\left(A_f^\prime\right)+D_d\left(B_f^\prime\right)+D_d\left(C_f^\prime\right)\right)\text{,}
\]
to be the \emph{width of $f$ with respect to $d$}.
\end{defn}
\section{Constructing new link maps}\label{chap6}

\subsection{Annular link maps}\normalfont

Our inspiration for Theorem \ref{R1} was to imagine taking the connect sum of two-component link maps  but around one of the tubes place an unknotted sphere such that the tube links this sphere. Our goal is to formalise this idea and provide a formula for $\widetilde{\sigma}^3$ for three-component link maps constructed via this method. We now introduce some formalism.

Let $D_1,\ldots D_n$ be disjoint embedded discs in $B^3$ with boundary $C_i$ such that $C_i\cap \partial B^3=\emptyset$ for all $i$. Additionally have all $C_i$ lie inside $\mathbb{R}^2\times \{0\}\cap B^3$ and are not nested in the plane.
\begin{defn}\normalfont
A \emph{$n$-component $2$-string link} is a smooth/topologically flat proper embedding of 
\[
\coprod_{i=1}^nS^1\times \left[0, 1\right]\rightarrow B^3\times \left[0,1\right],
\]
such that the image of each annulus is bounded by $C_i\times \{0\}$ and $C_i\times \{1\}$, with compatible orientation. An \emph{Annular link map} is defined similarly but we allow this map to be an immersion with self-intersections, on the same component, in the interior of each annulus. 
\end{defn}
We consider annular link maps up to link homotopy i.e. a  homotopy through annular link map. Denote the set of three-component annular up to link homotopy by $\SL_{2,2,2}^4$. Let $\ESL_{2,2,2}^4\subset \SL_{2,2,2}^4$ be the subgroup of three-component annular link maps which are link homotopic to a topologically flat embedded annular link map. An embedded annular link map which maps each component as $(p,t)\mapsto (f_t(p), t)$, where each $f_t$ is an embedding for all $t\in I$, is called a \emph{pure braid} .

We construct a link map from an annular link map by taking $B^3\times I$ and gluing along another $B^3\times I$ along $S^2\times I$ giving $S^3\times I=B^3\times I\cup_{S^2\times I}B^3\times I$. We now cap off both ends of $S^3\times I$ with $D^4$ and cap off each end of the annulus with a slice disc which is link homotopic to the collection of disjoint $D_i$, giving rise to a map $\SL_{2,2,2}^4\rightarrow \LM_{2,2,2}^4$. If $X\in\SL_{2,2,2}^4$ denote its \emph{link map closure} by $f_{X}\in\LM_{2,2,2}^4$ .

The set of annular link maps can be equipped with a multiplication. Suppose we have two annular link maps $X, X^\prime:\coprod_{i=1}^nS^1\times\left[0,1\right]\rightarrow B^3\times \left[0,1\right]$. We define 
\[
\left(X\cdot X^\prime\right)(p,t)=\begin{cases}
X(p,2t) & 0\leq t\leq \frac{1}{2}\\
X^\prime(p, 2t-1)& \frac{1}{2}<t\leq 1\text{.}\\
\end{cases}
\]

If we consider the set of $n$-component annular link maps up to link homotopy then $\SL_{2,2}^4$ becomes a group with multiplication given by
\[
\left[X\right]\cdot\left[X^\prime\right]:=\left[X\cdot X^\prime\right]\text{.}
\]
The inverse is given by considering the involution $r:B^3\times \left[0,1\right]\rightarrow B^3\times \left[0,1\right]$, with
\[
r(p,t)=(p, 1-t)
\]
and also let $Y:\coprod_{i=1}^nS^1\times [0,1]\rightarrow B^3\times \left[0,1\right]$ with
\[
Y(p,t)=X(p,1-t)\text{.}
\]
We then have a annular link map $r\circ Y:\coprod_{i=1}^nS^1\times \left[0,1\right]\rightarrow B^3\times \left[0,1\right]$ for which
\[
\left[r\circ Y\right]\cdot\left[X\right]=\left[X\right]\cdot\left[r\circ Y\right]= \left[O\right],
\]
where $O$ is the trivial annular link map. The proof that this gives an inverse can be found in \cite{M21}. We can show that this multiplication is non-abelian using the $\sigma^3$ invariant. We first define the following. For a annular link map $X$ we will denote its inverse by $\overline{X}$.

\begin{defn}
Let $B$ be a three-ball contained inside $\mathring{B}^3$.
Let $X$ be a three-component annular link map with one of the components contained in a $B\times I\subset B^3\times I$ where the other components do not intersect $B\times I$ and the image of this component is a $C_i\times I$ for some $i\in\{1,2,3\}$. The other two components are the image of a $JK$ construction which has not been capped off. Then we call $X$ a \emph{three-component JK construction}.
\end{defn}

\begin{figure}
    \centering
    \includegraphics[width=13cm]{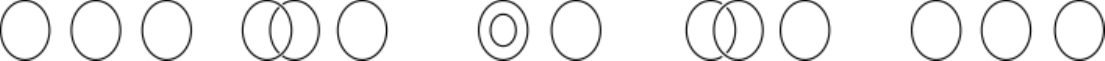}
    \caption{An example of an embedded annular link map.}
    \label{fig:PB}
\end{figure}

\begin{prop}[\cite{M21}]\label{commute}
The group of annular link maps up to link homotopy is non-abelian if $n\geq 3$.
\begin{proof}

Consider the three-component $JK$ construction inside $B^3\times I$ based on the  $JK$ construction using the Whitehead link, where the second component is trivial. Denote this by $X$, and similarly denote its inverse in $\SL_{2,2,2}^4$ by $\overline{X}$. Let $J$ be the annular link map described by Figure \ref{fig:PB}, inside the $B^3\times I$. Consider the stacking $XJ\overline{X}$, with link map closure $f_{XJ\overline{X}}$. We have
\[
\sigma^3(f_{XJ\overline{X}})=\left(z(s-1)+\bar{z}(\bar{s}-1),\, 0,\, x(1-u)+\bar{x}(1-\bar{u})\right)\text{,}
\]
as it is the same link map as in Theorem \ref{R1}.  Consider the link map closure $f_{X\overline{X}J}$. Notice that we can retract the tubes given by $J$ so that this is the same link map given by the closure of $X\overline{X}$. This gives a trivial link map \cite{M21} and thus $XJ\overline{X}$ and $X\bar{X}J$ are not equivalent.
\end{proof}
\end{prop}

\begin{defn}
Let $L\subset B^3$ be an oriented $n$-component link of circles. The \emph{group of singular concordances of $L$} is the set of flat topological immersions $\coprod_{i=1}^nS^1\times I\rightarrow B^3\times I$, which is a link map such that image of the boundary each annulus is $L\times\{0\}$ and $L\times\{1\}$ with compatible orientation and the only self-intersections are in the interior of each component, up to link homotopy. We denote this group by $C\left(L\right)$ and the group operation is given by stacking the concordances.
\end{defn}
To prove this has a group structure is similar to showing that the group of annular link maps also has a group structure.
\begin{prop}
Let $L$ and $L^\prime$ be concordant links in $B^3$. Then the groups $C(L)$ and $C(L^\prime)$ are isomorphic.
\begin{proof}
A similar proof is given in \cite{M21}. Let $Y$ be a concordance between $L^\prime$ and $L$ and let $X\in C\left(L\right)$ and $X^\prime\in C\left(L^\prime\right)$. We have two homomorphisms 
$\xi:C(L)\rightarrow C(L^\prime)$, $\xi^\prime:C(L^\prime)\rightarrow C(L)$ defined by
\[
\xi(X)=YX\overline{Y} \text{ and }
\xi^\prime(X^\prime)=\overline{Y}X^\prime Y.
\]
We have 
\[
\xi^\prime\circ \xi(X)=\overline{Y}YX\overline{Y}Y\sim_{lh} X\text{.}
\]
Hence, $\xi^\prime\circ \xi=\Id_{C(L)}$. Similarly, $\xi\circ \xi^\prime=\Id_{C(L^\prime)}$.
\end{proof}
\end{prop}

\begin{cor}[\cite{M21}]\label{commuteL}
Let $L$ be a $n$-component slice link. The group of singular concordances of $L$ up to link homotopy is non-abelian for $n\geq 3$.
\begin{proof}
By the previous proposition, the group of (singular) concordances of a slice link up to link homotopy is isomorphic to the group of annular link map. Hence, we only have to show that the group of annular link maps isn't abelian for $n\geq 3$. However, this is shown in Proposition \ref{commute}.
\end{proof}
\end{cor}

\subsection{A three-component annular link map invariant}

Both Proposition \ref{commute} and Corollary \ref{commuteL} were proven in \cite{M21} by showing that the subgroup of embedded annular link maps is non-abelian, using the longitudes of each component; whereas we studied the self-intersections of annular link maps. 

We now create an invariant of annular link maps which incorporates the longitude information but also records self-intersection information similar to our invariant for link maps.

First we specify a choice of basing path for each annular link map with three components. Take the basepoint to be $\left(s_0,0\right)\in B^3\times I$ and let $p_i$ be a point in a tubular neighbourhood of the $i$th component of the unlink $O_3$ inside $B^3\times\{0\}$. We define a meridian for the $i$th component by taking a path from $(s_0, 0)$ to $p_i$ with the curve lying only inside $B^3\times\{0\}$ and disjoint from the discs $D_i$. Finally, travel along a small loop around the $i$th component, such that the linking number is positive, and return along the path which took you to $p_i$. We call the meridians  $x$, $y$ and $z$ for the first second and third components respectively. The basing curves for the $i$th component is given by taking the same path to $p_i$ for the meridians and then travelling radially down the fibre of the normal bundle to the $i$th component.

\begin{defn}
We call such a collection of meridians and basing curves a \emph{ basing system} for a annular link map.
\end{defn}

We defined basing systems for annular link map above in terms of the unlink in the $B^3\times \{0\}$ time slice. However, if $B^3\times\{t\}\cap X=O_3\times \{t\}$ at $t\in I$  we will say there is a basing system there too, where we take the basepoint to be $\left(s_0, t\right)$ for these other basing systems. Hence, we will often talk about a \emph{basing system at $t\in I$}.

Annular link maps specify an automorphism of the lower central series quotients of the fundamental group of the complement of the unlink. To prove this we will first prove the following Lemma.

\begin{lem}\label{stringlinkmil}
Let $X$ be the image of a generic  $n$-component annular link map inside $B^3\times I$. Then $M\pi_1\left(B^3\times I\setminus X\right)$ is isomorphic to the Milnor free group on $n$-generators.
\begin{proof}
We can apply finger moves to the annular link map such that $\pi_1\left(B^3\times I\setminus X\right)=M\pi_1\left(B^3\times I\setminus X\right)$, which are normally generated by meridians of the unlink in $B^3\times\{1\}$. Using Seifert-Van Kampen we have
\[
\pi_1\left(S^3\times I\setminus X\right)=\pi_1\left(B^3\times I\setminus X\right)*\pi_1\left(B^3\times I\right)=\pi_1\left(B^3\times I\setminus X\right)\text{,}
\]
where $S^3\times I\setminus X=B^3\times I\setminus X\cup_{S^2\times I}B^3\times I$. Capping off $S^3\times\{1\}$ with $D^4$ and choosing slice discs for each component which are link homotopic to the standard embedding of discs bounded by the $n$-component unlink call these discs $\Delta$. 

This gives a collection of singular but disjoint slice discs for the unlink in $D^4$. We call this collection $\Delta^\prime$. By Seifert-Van Kampen, we have
\[
\pi_1\left(D^4\setminus \Delta^\prime\right)=\pi_1\left(S^3\times I\setminus X\right)*_{F(n)}\pi_1(D^4\setminus \Delta)\text{,}
\]
Since $\pi_1(D^4\setminus \Delta)\cong F(n)$ where the generators of $F(n)$ is the choice of meridians of the unlink we made previously. However, the amalgamation identifies the generators of $\pi_1(S^3\times I\setminus X)$ and $\pi_1(D^4\setminus \Delta)$, Hence
\[
\pi_1\left(B^3\times I\setminus X\right)\cong \pi_1\left(D^4\setminus\Delta^\prime\right)
\]
By a result in \cite{DT}, we have that $M\pi_1\left(D^4\setminus\Delta^\prime\right)\cong MF(n)$. This proves the result.
\end{proof}
\end{lem}

\begin{lem}\label{isopath}
Let $X$ be a Annular link map. Then the inclusion maps 
\[
j_i:\left(B^3\times\{i\}\setminus O_3, \left(s_0, i\right)\right)\rightarrow \left(B^3\times I\setminus X, \left(s_0, i\right)\right),
\]
where $i=0,1$, induces an isomorphism on the Milnor groups of the fundamental groups of each space.
\begin{proof}
By Lemma \ref{stringlinkmil}, we have that the Milnor group of the complement of $X$ inside $B^3\times I$ is the Milnor free group on three-generators. It is clear that since we are sending generators to generators and both are the free Milnor object of the same number of generators we have an isomorphism.

\end{proof}
\end{lem}
We define
\begin{align*}
\pi_{O_3}&:=\pi_1\left(B^3\times \{0\}\setminus O_3, (s_0, 0)\right),\\
\pi_{O^\prime_3}&:=\pi_1\left(B^3\times \{1\}\setminus O_3, (s_0,1 )\right),\\
\pi^i_X&:=\pi_1\left(B^3\times I\setminus X, (s_0, i)\right)\text{.}
\end{align*}
By Lemma \ref{isopath}, $\left(j_0\right)_*:M\pi_{O_3}\rightarrow M\pi^0_{C}$ and
 $\left(j_1\right)_*:M\pi_{O^\prime_3}\rightarrow M\pi^1_{C}$
are isomorphisms. Let $\alpha:[0,1]\rightarrow \left(B^3\times I\setminus X\right)$ be given by $\alpha(t)=(s_0, t)$\text{.}

Let $\psi_\alpha:\pi_X^0\rightarrow \pi_X^1$ be the isomorphism defined by $\gamma\mapsto \bar{\alpha}\cdot\gamma\cdot\alpha$. Hence, we have two isomorphisms
\begin{align*}
(j_1)^{-1}_*\circ \psi_\alpha\circ(j_0)_*:M\pi_{O_3}&\rightarrow M\pi_{O^\prime_3},\\
(j_0)^{-1}_*\circ \psi_{\alpha}^{-1}\circ(j_1)_*:M\pi_{O^\prime_3}&\rightarrow M\pi_{O_3}\text{.}
\end{align*}

A more geometric description of the content of these isomorphisms is to consider the basing system for the annular link map lying in the $t=0$ and another basing system lying in the $t=1$ slice. Let the $m_i$ and the $m_i^\prime$ be meridians for the $i$th component at each time slice $t=0$ and $t=1$ respectively. We can map
\[
m^\prime_i\mapsto (j_0)^{-1}_*\circ \psi_{\alpha}^{-1}\circ(j_1)_*(m^\prime_i)=g_i^{-1}m_ig_i\in M\pi_{O_3}
\]
for some $g_i\in M\pi_{0_3}$. This $g_i$ is represented by a longitude of the $i$th component of the annular link map which we can decompose in terms of the meridians of the basing system at $t=0$. 

\begin{defn}
Let $X$ be an annular link map with a basing system at both $t=0$ and $t=1$. Then the \emph{$i$th longitude} is a loop based at $(s_0, 0)$ and is defined by concatenating the following paths:

\begin{itemize}
    \item use the basing system at $t=0$ travel to from $(s_0, 0)$ to $(p_i, 0)$.\\
    \item Next, take a path in a regular neighbourhood of the $i$th component, which at no point lies in the fibre above a double point, and travels between $(p_i, 0)$ to $(p_i, 1)$.\\
    \item Use the basing system at $t=1$ to travel from $(p_i, 1)$ to $(s_0, 1)$.\\
    \item Finally, travel along the path  $\overline{\alpha}:\left[0,1\right]\rightarrow B^3\times I$ defined by $\overline{\alpha}(t)=(s_0, 1-t)$.
\end{itemize}
\end{defn}

Once we have fixed a choice of basing path and meridians at either end of the cylinder there are, up to homotopy, $\mathbb{Z}^2$ many choices of longitude of the $i$th component, determined by the linking number of the longitude with the $i$th component and how much ``wraps around'' the annulus. To deal with this indeterminacy we always choose longitudes with this linking number equal to zero. The indeterminacy given by wrapping round the cylinder does not affect the end result up to homotopy, as we have an unlink in the boundary.

\begin{rem}
For three-component JK constructions we will always assume that the longitudes of the $i$th component is trivial in $\Gamma_i$.
\end{rem}

Let $\Aut(MF(3))$ be the group of automorphisms of the free Milnor group on three generators, with the usual group structure given by composition of functions. Recall that $K=\mathbb{Z}\Gamma_x\times \mathbb{Z}\Gamma_y\times\mathbb{Z}\Gamma_z$. We will define a map $\Phi:\Aut(MF(3))\rightarrow \Aut(K)$ where $\Aut(K)$ is the group of $\mathbb{Z}$ linear ring automorphisms of $K$. Notice that for $\psi\in\Aut(MF(3))$ there exists a map $\psi_i:\Gamma_i\rightarrow \Gamma_i$ such that
\[
\begin{tikzcd}
MF(3)\arrow[r, "\psi"]\arrow[d]&MF(3)\arrow[d]\\
\Gamma_i\arrow[r,"\psi_i"]&\Gamma_i\text{}
\end{tikzcd}
\]
commutes for each $i\in\{x,y,z\}$. We use the same notation for these maps to denote extension to the corresponding group ring.

Evaluating $\psi$ on each of the generators we have
\begin{align*}
\psi(x)&=\tau_x^{-1}x\tau_x,\\
\psi(y)&=\tau_y^{-1}y\tau_y,\text{ and }\\
\psi(z)&=\tau_z^{-1}z\tau_z\text{.}
\end{align*}
Denote the image of $\tau_i$ in the quotient group by $\left[\tau_i\right]\in \Gamma_i$.

Given $(a,b,c)\in K$, we define $\Phi:\Aut(MF(3))\rightarrow \Aut(K)$ to be 
\[
\Phi(\psi)(a,b,c)= \left(\left[\tau_x\right]\psi_x(a)\left[\tau_x\right]^{-1}, \left[\tau_y\right]\psi_y(b)\left[\tau_y\right]^{-1}, \left[\tau_z\right]\psi_z(c)\left[\tau_z\right]^{-1}\right)\text{.}
\]

\begin{lem}
The map $\Phi:\Aut(MF(3))\rightarrow \Aut(K)$ is a group homomorphism.
\begin{proof}
Let $\alpha,\beta\in \Aut{MF(3)}$. We wish to show that that $\Phi(\beta\circ\alpha)=\Phi(\beta)\circ\Phi(\alpha)$. We write 
\[
\Phi(\alpha)(a,b,c)=\left(\left[\tau_x\right]\alpha_x(a)\left[\tau_x\right]^{-1},\left[\tau_y\right]\alpha_y(b)\left[\tau_y\right]^{-1}, \left[\tau_z\right]\alpha_z\left(c\right)\left[\tau_z\right]^{-1}\right)
\]
and 
\[
\Phi(\beta)(a,b,c)=\left(\left[\upsilon_x\right]\beta_x(a)\left[\upsilon_x\right]^{-1}, \left[\upsilon_y\right]\beta_y(b)\left[\upsilon_y\right]^{-1},\left[\upsilon_z\right]\beta_z(c)\left[\upsilon_z\right]^{-1}\right)\text{.} 
\]
We calculate 
\[
\beta\circ\alpha(x)=\beta\left(\tau_x^{-1}x\tau_x\right)=\beta(\tau^{-1}_x)\beta(x)\beta(\tau_x)=\beta(\tau_x^{-1})\upsilon_x^{-1}x\upsilon_x\beta(\tau_x),
\]
 Defining similarly on $y$ and $z$, we have
\begin{multline*}
\Phi(\beta\circ\alpha)(a,b,c)=\Big(\left[\upsilon_x\right]\beta_x(\left[\tau_x\right])\left(\beta\circ\alpha\right)_x\left(a\right)\beta_x(\left[\tau_x\right]^{-1})\left[\upsilon_x\right]^{-1},\\\left[\upsilon_y\right]\beta_y(\left[\tau_y\right])\left(\beta\circ\alpha\right)_y\left(b\right)\beta_y(\left[\tau_y\right]^{-1})\left[\upsilon_y\right]^{-1},\,
\left[\upsilon_z\right]\beta_z(\left[\tau_z\right])\left(\beta\circ\alpha\right)_z\left(a\right)\beta_z(\left[\tau_z\right]^{-1})\left[\upsilon_z\right]^{-1}\Big)\\
=\Big(\left[\upsilon_x\right]\beta_x(\left[\tau_x\right])\beta_x\left(\alpha_x\left(a\right)\right)\beta_x(\left[\tau_x\right]^{-1})\left[\upsilon_x\right]^{-1},\\\left[\upsilon_y\right]\beta_y(\left[\tau_y\right])\beta_y\left(\alpha_y\left(a\right)\right)\beta_y(\left[\tau_y\right]^{-1})\left[\upsilon_y\right]^{-1},\,
\left[\upsilon_z\right]\beta_z(\left[\tau_z\right])\beta_z\left(\alpha_z\left(a\right)\right)\beta_z(\left[\tau_z\right]^{-1})\left[\upsilon_z\right]^{-1}\Big)\\
\end{multline*}
As $(\beta\circ\alpha)_i=\beta_i\circ\alpha_i$ we have
\begin{align*}
\Phi(\beta\circ\alpha)(a,b,c)&=\Phi(\beta)\left(\left(\left[\tau_x\right]\alpha_x(a)\left[\tau_x\right]^{-1},\left[\tau_y\right]\alpha_y(b)\left[\tau_y\right]^{-1}, \left[\tau_z\right]\alpha_z\left(c\right)\left[\tau_z\right]^{-1}\right)\right)\\
&=\Phi(\beta)\circ\Phi(\alpha)\left(\left(a,b,c\right)\right)\text{,}
\end{align*}
as required. 
\end{proof}
\end{lem}

Define
\[
L:=K\rtimes\Aut(MF(3))
\]
where we treat K an abelian group given by addition. The group multiplication in $L$ is defined to be 
\[
\left(x_1, y_1\right)\left(x_2, y_2\right):=\left(x_1+\Phi(y_1)\left(x_2\right), y_1\circ y_2\right)\text{,}
\]
a semi-product, where $\Phi$ determines the action of $MF(3)$ on $K$.

A link homotopy invariant for annular link maps is given by the following: let $X$ be a annular link map in $B^3\times I$, let $X_x$, $X_y$ and $X_z$ be the $x$, $y$ and $z$ components respectively. Define 
\[
\Theta_x(X)=M\lambda\left(X_x, X_x\right)\in\mathbb{Z}\Gamma_x
\]
where the normal push-off which is used keeps the boundary of $X_x$ inside $\mathring{B}^3\times\{0,1\}$, and the interior of $X_x$ inside $\mathring{B}^3\times \mathring{I}$, and we use the basing system at $B^3\times \{0\}$ to compute $\lambda\left(X_x, X_y\right)$ and to specify an isomorphism to $\mathbb{Z}\Gamma_x$. We define $\Theta_y(X)$ and $\Theta_z(X)$ similarly. Let $\eta(X)$ be the element of $\Aut(MF(3))$ given by considering the isomorphisms $M\pi_1(B^3\times \{1\}\setminus O_3, (s_0, 1))\rightarrow M\pi_1(B^3\times \{0\}\setminus O_3, (s_0, 0))$ given by $X$.

\begin{defn}
Let $X$ be an annular link map. We define
\[
\Theta(X):=\left(\left(\Theta_x(X),\Theta_y\left(X\right), \Theta_z\left(X\right)\right),\eta(X)\right)\in L\text{.}
\]
\end{defn}

\begin{prop}
Let $X$ and $X^\prime$ be link homotopic annular link maps. Then
\[
\Theta\left(X\right)=\Theta\left(X^\prime\right)\text{.}
\]

\begin{proof}
The proof for the factors involving $\Theta_x$, $\Theta_y$ and $\Theta_z$ are the same as for the case of three-component link maps. We must show that the longitudes are unaffected by link homotopy. However, this is clear because the link homotopy is a concatenations of isotopy, cusps, Whitney moves and finger moves. These moves do not affect the longitude. This completes the proof.
\end{proof}
\end{prop}

\begin{theo}
The map 
\[
\Theta:\SL_{2,2,2}^4\rightarrow L,
\]
is a group homomorphism.
\begin{proof}
We must show that
\begin{multline*}
\Theta(X\cdot X^\prime)=\Big(\left(\Theta_x(X), \Theta_y(X), \Theta_z(X)\right)+\Phi(\eta(X))\left(\Theta_x(X^\prime), \Theta_y(X^\prime), \Theta_z(X^\prime)\right), \eta(X)\circ\eta(X^\prime)\Big)\text{.}
\end{multline*}

We will first show that the stacking operation on annular link maps corresponds to composition in $\Aut(MF(3))$. Let $g\in\pi_1(B^3\times \{1\}\setminus O_3)\cong MF(3)$ and let $\alpha_{0}, \alpha_{\frac{1}{2}}:I\rightarrow B^3\times I\setminus \left(X\cdot X^\prime\right)$
where 
\[
\alpha_0(t)=\left(s_0, \frac{t}{2}\right)
\]
and 
\[
\alpha_{\frac{1}{2}}(t)=\left(s_0,\frac{1+t}{2}\right).
\]
Rebase $g$ to get $\alpha_{\frac{1}{2}}\cdot g\cdot \overline{\alpha}_{\frac{1}{2}}$ and think of it as an element of $\pi_1\left(B^3\times \left[\frac{1}{2}, 1\right]\setminus X^\prime, \left(s_0, \frac{1}{2}\right)\right)$. We then pull this back using the induced map on the Milnor group of the fundamental groups given by $B^3\times\{\frac{1}{2}\}\setminus O_3 \hookrightarrow B^3\times I\setminus X$. This gives us an element of $MF(3)\cong M\pi_1\left(B^3\times\{\frac{1}{2}\}, \left(s_0, \frac{1}{2}\right)\right)$ and equal to $\eta(X^\prime)(g)$. We follow a similar argument rebasing using $\alpha_0$ so that the induced map from the stacking is $\eta(X)\circ \eta(X^\prime)$, as required.

We need to check the effect of stacking on the self-intersection information. We may assume that the intersections on $X$ and $X^\prime$ have cancelling sign. The intersection on the $X$ region is not affected since those intersections will use the same meridians and basing paths to decompose the group elements and thus their contributions are not affected. We must check that the values of the intersections in the $X^\prime$ region. First notice that the $\Theta_x(X^\prime), \Theta_y(X^\prime), \Theta_z(X^\prime)$ are equal to the sum of the intersections on the $X^\prime$ region using the basing system at $B^3\times \{\frac{1}{2}\}$. We now must rebase and describe the group elements using the basing system at $B^3\times \{0\}$. To calculate $\Theta_x(X^\prime)$ intersections using the meridians of the basing system at $B^3\times \{0\}$ we apply the automorphism on $\Gamma_x$ induced by $\eta(X)$. We then must take into account our choice of basing path. However, this is just conjugation by the longitude of the $x$ component of $X$ as required. The proofs for the $y$ and $z$ components are analogous.
\end{proof}
\end{theo}
From the construction of $\Theta$ we have the following result.
\begin{theo}
The following diagram commutes
\begin{equation}\label{invcom}
\begin{tikzcd}
\SL_{2,2,2}^4\arrow[r, twoheadrightarrow]\arrow[d, "\Theta"]&\LM_{{2,2,2}}^4\arrow[r, twoheadrightarrow, "i"]\arrow[d, "\widetilde{\sigma}^3"]&\LM_{2,2}^4\arrow[d,"\sigma"]\\
L\arrow[r]&\widetilde{K}\arrow[r, "p_i"]&\left(\mathbb{Z}\left[\mathbb{Z}\right]\right)^2\text{.}
\end{tikzcd}
\end{equation}
\end{theo}

\begin{prop}
Let $X$ be an embedded three-component annular link map such that $\Theta(X)=(a, \phi)$ where $a\neq 0$. Then $X$ is not link homotopic to a topologically flat embedding.
\end{prop}

\subsection{Describing link maps}
We now wish to construct link maps by giving a description of a annular link maps and then calculating $\Theta$ to determine the values of $\sigma_3$ and $\tilde{\sigma}_3$ after taking the closure.  

\begin{defn}
We call an embedded annular link a \emph{ribbon braid} if it is the trace of an isotopy in $B^3$, where for each time slice we have an unlink with each component lying in a $B^3\cap \mathbb{R}^2\times\{q\}$ for some $q\in\mathbb{R}^2$.
\end{defn}
\begin{rem}
This is different from how ribbon braids are defined in \cite{A}, where the components lying in a plane parallel to each other condition is loosened.
\end{rem}

Let $p_1,\ldots, p_n$ be pairwise distinct elements in $\mathring{I}$. Let $f:\coprod_{i=1}^nI_i\hookrightarrow I\times I$ with $I_i=\left[0,1\right]$ and $f(j)=\{p_i\}\times\{j\}$ with $j\in \partial I_i$ be a topologically flat, transverse immersion such that $f(i)=(x, i)$ for some $x\in I$. For all double points $p$, place a partial ordering on each elements in the pre-image. If we can say that one element in the pre-image is less than or greater than the other - using the partial ordering - then we delete a small neighbourhood of the lower point, if we cannot this is called a \emph{welded crossing} and we place a circle around this crossing. We call this a \emph{virtual braid diagram}. 

\begin{defn}
We call the set of virtual braid diagrams, up to standard Reidmeister moves; virtual Reidemeister moves described in Figures \ref{fig:VR1}, \ref{fig:VR2}, \ref{fig:VR3}, \ref{fig:VR4}; and the over crossing move shown in Figure \ref{fig:overcommute}, \emph{welded braids}. 
\end{defn}

\begin{figure}
    \centering
    \includegraphics{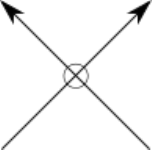}
    \caption{A welded crossing}
    \label{fig:weldedcross}
\end{figure}
\begin{figure}
    \centering
    \includegraphics{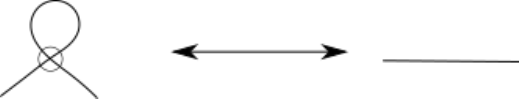}
    \caption{The first virtual Reidemeister move.}
    \label{fig:VR1}
\end{figure}
\begin{figure}
    \centering
    \includegraphics{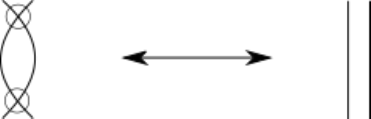}
    \caption{The second virtual Reidemeister move.}
    \label{fig:VR2}
\end{figure}

\begin{figure}
    \centering
    \includegraphics{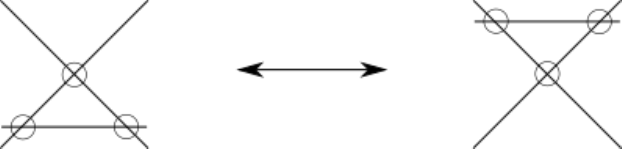}
    \caption{The third virtual Reidemeister move.}
    \label{fig:VR3}
\end{figure}

\begin{figure}
    \centering
    \includegraphics{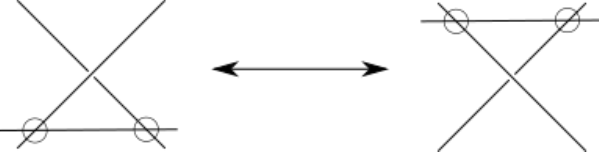}
    \caption{The fourth virtual Reidemeister move.}
    \label{fig:VR4}
\end{figure}
\begin{figure}
    \centering
    \includegraphics{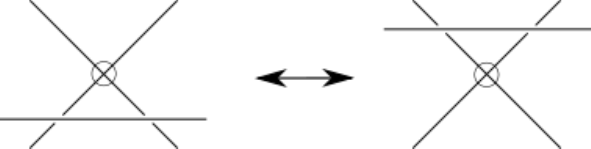}
    \caption{Overcrossing on welded braids.}
    \label{fig:overcommute}
\end{figure}

We can stack welded braids together to produce new welded braids. This gives a group structure on the set of welded braids. Denote this group $\mathfrak{W}$. There is a well defined map from $\mathfrak{W}$ to ribbon braids called the tube map which is discussed in \cite{A} and due to the results of Brendle and Hatcher \cite{HB} this is an isomorphism. Hence we can represent ribbon braids by a welded braid diagram.

\begin{defn}
Let $L$ be a welded braid, where each strand is oriented from $I\times \{0\}$ to $I\times\{1\}$, an \emph{overstrand} is a piece of arc with each endpoint being either the boundary or the underpass of a classical crossing, with no classical crossing in the interior of this piece of the arc. Let $O(L)$ be the set of overstrands of $L$ and let $C(L)$ be the set of classical crossings of $L$. For $c\in C(L)$ we define the following:
\begin{itemize}
    \item $\epsilon_c$ is the sign of the crossing at $c$
    \item $s^0_c$ is the overstrand containing the highest pre-image at the crossing,
    \item $s_c^-$ is the overstrand whose boundary element is the lowest preimage of $c$ and the orientation of the strand point into the crossing,
    \item $s_c^+$ is the overstrand whose boundary element is the lowest preimage of $c$ and the orientation of the strand point out of the crossing.
\end{itemize}
The fundamental group of $L$ is defined to be
\[
\pi_1\left(L\right):=\left<O\left(L\right)\mid s^+_c=\left(s^-_c\right)^{\left(s^0_c\right)^{\epsilon_c}} \forall c\in C(L)\right>\text{.}
\]
\end{defn}

 Brendle and Hatcher, in \cite{HB}, showed that the tube map provides the following isomorphism
\[
\pi_1\left(L\right)\cong\pi_1\left(B^3\times I\setminus \text{Tube}(L)\right).
\]

In fact, there is a one-to-one correspondence between the generators and relations, coming from classical crossing of the welded braid and Wirtinger presentation given by a broken surface diagram for the image of $L$ under the tube map \footnote{A comprehensive account of broken surface diagrams can be found in \cite{A}.}. Thus, we can use welded braids to compute longitudes of elements and compute the isomorphisms of $\Aut(MF(3)$ induced by a ribbon braid, which makes constructing examples easier. It follows that we can use these diagrams to specify any automorphisms of $MF(3)$ using results from \cite{A}.

\begin{exam}
Consider the welded braid as in Figure \ref{fig:easycon.pdf}. Then the longitude of the $x$-component is overstrand of the $y$ component. Hence the element $\psi\in~\Aut(MF(3))$ described by this welded braid is given by
\[
\psi(x)=y^{-1}xy,
\]
where $\psi(y)=y$ and $\psi(z)=z$.
\end{exam}

\begin{defn}
Let $f\in \LM_{2,2}^4$ and $g\in F/F_3$ then we define  $\sigma_i^g(f)$ to be the Laurent polynomial $\sigma_i(f)$ evaluated on $g$.
\end{defn}

\begin{defn}
Let $X$ be a three-component annular link map and $i\in\{1,2,3\}$. Then $X^i$ is a two-component annular link map which is $X$ with the $i$th component removed.
\end{defn}

\begin{figure}
 \includegraphics[scale=0.75]{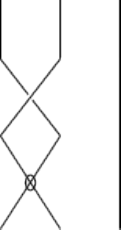}

\caption[A welded braid]{A welded braid description of the ribbon braid described in \ref{fig:PB}, where the $x$-component is starts and ends on the left, the $y$-component starts and ends in the centre, and the $z$-component is on the right.}
\label{fig:easycon.pdf}
\end{figure}

 Let $X:=X_1J_1X_2\ldots J_{n-1}X_n$ be a stacking of JK constructions. Furthermore, we define:
\begin{itemize}
    \item Let $k_{i}$ and $k_{i}^\prime$ be the difference of the number of positive and negative classical crossing of a longitude of the $x$, with the $y$ and $z$ components respectively, for the Welded Braid which is mapped to the  substack $J_1\ldots J_{i-1}$ under the tube map, when $i>1$. We set $k_1=k^\prime_1=0$.
    \item Let $l_{i}$ and $l_{i}^\prime$ be the difference of the number of positive and negative classical crossing of a longitude of the $y$ component, with the $z$ and $x$ components respectively, for the welded braid which is mapped to the substack $J_1\ldots J_{i-1}$ under the Tube map, when $i>1$. We set $l_1=l^\prime_1=0$.
    \item Let $m_{i}$ and $m_{i}^\prime$ be the difference of the number of positive and negative classical crossing of a longitude of the $z$ component, with the $x$ and $y$ components respectively, for the welded braid which is mapped to the substack $J_1\ldots J_{i-1}$ under the tube map, when $i>1$. We set $m_1=m^\prime_1=0$.
\end{itemize}
\begin{theo}\label{main1}
Let $f_X$ be a link map arising from the stacking $X$. Then, for some choice of basing path, we have
\begin{align*}
\sigma_x\left(f_X\right)&=\sum_{i=1}^nz^{k^\prime_i-l_i}\sigma_1^y\left(f_{X^{3}_i}\right)z^{-k^\prime_i+l_i}+y^{k_i-m^\prime_i}\sigma_1^z\left(f_{X^{2}_i}\right)y^{-k_i+m^\prime_i}\\ \sigma_y(f_X)&=\sum_{i=1}^nz^{l_i-k^\prime_i}\sigma_2^x\left(f_{X_i^3}\right)z^{-l_i+k^\prime_i}+x^{l^\prime_i-m_i}\sigma_1^z\left(f_{X_i^1}\right)x^{-l^\prime_i+m_i}\\
\sigma_z\left(f_X\right)&=\sum_{i=1}^ny^{m^\prime_i-k_i}\sigma_2^x\left(f_{X^{2}_i}\right)y^{-m^\prime_i+k_i}+x^{m_i-l_i^\prime}\sigma_2^y\left(f_{X^{1}_i}\right)x^{-m_i+l_i^\prime}.
\end{align*}

\begin{proof}
In this proof, we will use $\mu$ for our calculation instead of $\lambda$. Hence, we will assume that the sum of the signed intersections are zero.

 Recall that $\eta(X_i)=\Id$ as $X_i$ is a three-component JK construction and for each $J_i$ we have $\theta_p(J_i)=0$ for $p\in\{x,y,z\}$. We first calculate $\Theta$ and then using commutative diagram \eqref{invcom} we will map to the image of the link map closure under $\sigma^3$.  As $\Theta$ is a homomorphism 
\begin{multline*}
\Theta(X)=\Big(\left(\theta_x(X_1),\theta_y(X_1),\theta_z(X_1)\right)\\+\sum_{i=2}^n\Phi\left(\eta\left(\Pi_{j=1}^{i-1}J_j\right)\right)\left(\theta_x(X_i),\theta_y(X_i),\theta_z(X_i)\right),
\eta(J_1\ldots J_{n-1})\Big)\text{.}
\end{multline*}
 
 We now calculate $\Theta_x(X_i)$, the calculations for $\theta_y(X_i)$ and $\theta_z(X_i)$ are similar. As $X_i$ is a three-component JK construction there are three cases depending on which component is embedded. In the case where the $x$ component is embedded we know that $\theta_x(X_i)=0$ and thus $\sigma_1\Big(f_{X_i^3}\Big)=0$ and $\sigma_1\Big(f_{X_i^2}\Big)=0$. Hence, the claim is true in this first case. The next case to consider is to suppose $x$ has self-intersections and $y$ is the embedded component, split from the other two components. For each intersection on the $x$ component we associate an element of $\Gamma_x$. Since the $y$ component is split from the rest of the components, the group elements of the intersections of the $x$ components are powers of the $z$ meridian coming from the basing system at $B^3\times \{\frac{1}{2^{1+2n-2i}}\}$. Taking the sum over all the intersections of the $x$ component and considering their values inside $\Gamma_x/A_y$ and we have $\sigma_1(f_{X_i^2})$. Thus $\theta_x(X_i)=\sigma^z_1(f_{X_i^2})$. Since $\sigma_1^y(f_{X_i^3})=0$ we have
 \[
 \Theta_x(X_i)=\sigma^z_1(f_{X_i^2})+\sigma_1^y(f_{X_i^3})\text{.}
 \]
 A similar argument shows that this equation holds in the remaining case. We now show that 
\begin{multline*}
 \Phi(\eta(J_1\ldots J_{i-1}))\left(\left(\theta_x(X_i),\theta_y\left(X_i\right), \theta_z(X_i)\right)\right)\\
 =\Big(z^{k^\prime_i-l_i}\sigma_1^y\left(f_{X^{3}_i}\right)z^{-k^\prime_i+l_i}+y^{k_i-m^\prime_i}\sigma_1^z\left(f_{X^{2}_i}\right)y^{-k_i+m^\prime_i},\\ z^{l_i-k^\prime_i}\sigma_2^x\left(f_{X_i^3}\right)z^{-l_i+k^\prime_i}+x^{l^\prime_i-m_i}\sigma_1^z\left(f_{X_i^1}\right)x^{-l^\prime_i+m_i}, \\y^{m^\prime_i-k_i}\sigma_2^x\left(f_{X^{2}_i}\right)y^{-m^\prime_i+k_i}+x^{m_i-l_i^\prime}\sigma_2^y\left(f_{X^{1}_i}\right)x^{-m_i+l_i^\prime}\Big)\text{.}
\end{multline*}
Since $\Phi\left(\eta(J_1\ldots J_{i-1})\right)$ can be determined by $\tau_x^{i}\in\Gamma_x$, $\tau_y^i\in\Gamma_y$, $\tau_z^i\in\Gamma_z$, the longitudes of the $x$,$y$ and $z$ longitudes of $J_1\ldots J_{i-1}$. We know that as the longitude of a component welded braid which represents the $J_1\ldots J_{i-1}$ is equivalent to the longitude of the corresponding component of $J_1\ldots J_{i-1}$. Given $\tau_x=y^{k_i}z^{k_i^\prime}s^{p^x_i}$, $\tau_y=z^{l_i}x^{l_i^\prime}t^{p_i^y}$ and $\tau_z=x^{m_i}y^{m_i^\prime}u^{p_i^z}$, where $p_i^x, p_i^y, p_i^z\in\mathbb{Z}$. Thus evaluating $\Phi\left(\eta(J_1\ldots J_n)\right)$ on $\theta_x(X_i)$ gives
\begin{align*}
\tau_x\sigma_1^{z^{-l_i}yz^{l_i}}(f_{X_i^3})\tau^{-1}_x+\tau_x\sigma_1^{z^{-m_i\prime}yz^{m_i^\prime}} \left(f_{X_i^2}\right)\tau_x^{-1}
&=z^{k^\prime_i}\sigma_1^{z^{-l_i}yz^{l_i}}(f_{X_i^3})z^{-k^\prime_i}+y^{k_i}\sigma_1^{z^{-m_i\prime}yz^{m_i^\prime}} \left(f_{X_i^2}\right)y^{-k_i}\\
&=z^{k^\prime_i-l_i}\sigma_1^{y}\left(f_{X_i^3}\right)z^{-k_i^\prime+l_i}+y^{k_i-m_i^\prime}\sigma_1^{y}\left(f_{X_i^3}\right)z^{-k_i+m_i^\prime}
\end{align*}
A similar result shows the required effect on $\theta_y(X_i)$ and $\theta_z(X_i)$. Using commutative diagram \eqref{invcom}, we have the result.
\end{proof}
\end{theo}
 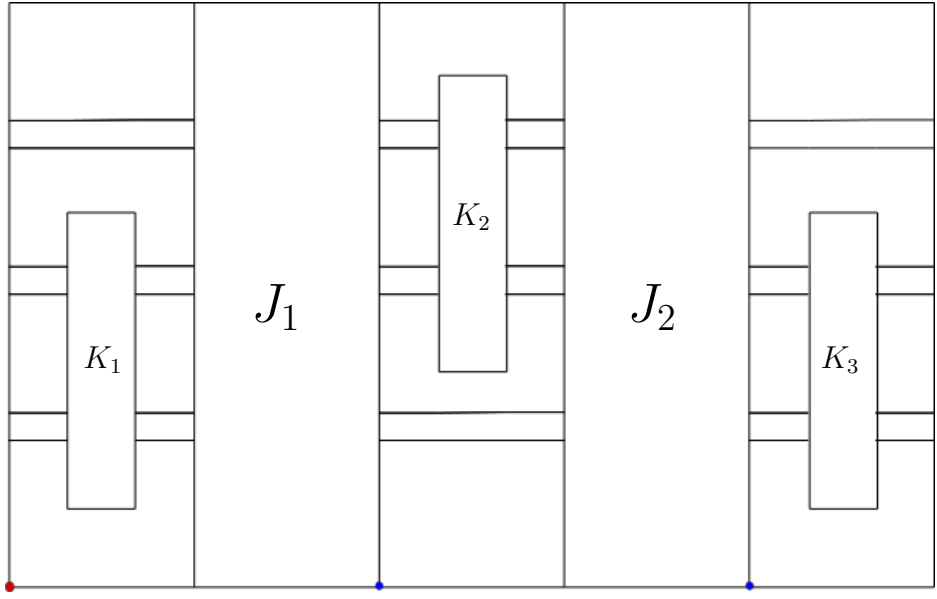
\begin{figure}
\begin{tikzpicture}
\node [
    above right,
    inner sep=0] (image) at (0,0) {\includegraphics[scale=0.75]{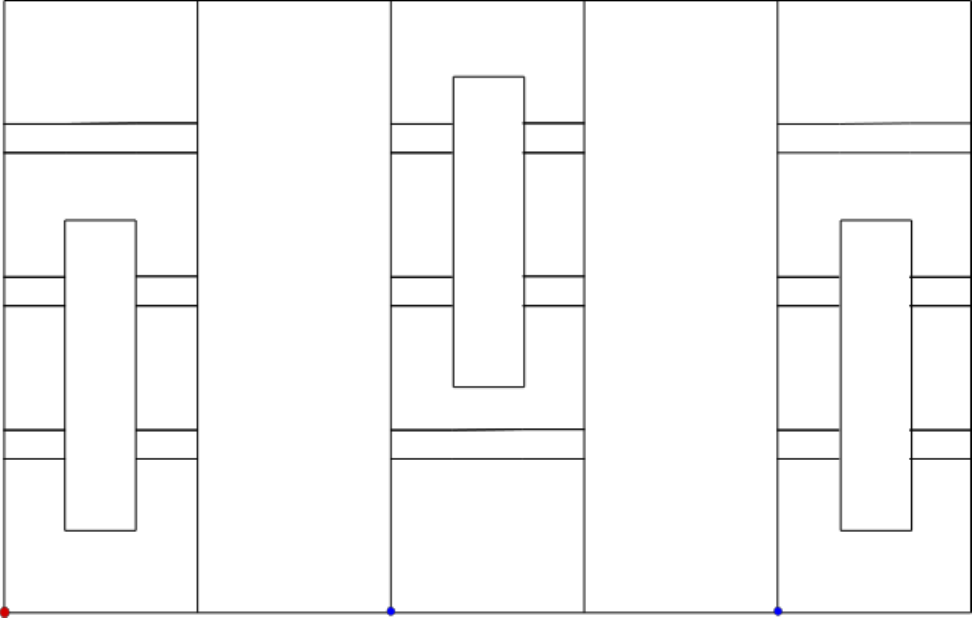}};
     \node at (1.3,3.1){\large $K_1$};
     \node at (3.6,3.8){\huge $J_1$};
     \node at (6.2,5){\large $K_2$};
     \node at (8.6,3.8){\huge $J_2$};
     \node at (11.1, 3.1){\large $K_3$};
\end{tikzpicture}
    \caption[Schematic diagram of the proof of Theorem \ref{main1}]{A schematic for the proof of Theorem \ref{main1} where the red point represents the basepoint we want to make our calculations from and the blue represent the basepoints we use to make our initial calculations of the intersections of $K_i$ before we then rebase and use the isomorphism induced by the stack of singular pure braids to the left of the basepoint.}
    \label{fig:schemtatic}
\end{figure}

 This result gives the image of $\widetilde{\sigma}^3$ for three-component link maps which are the result of connect sum of link maps which have two components which non-trivially link and another component split from those two which is trivial. One may wonder if all values of $k_i, k_i^\prime, l_i, l_i^\prime, m_i, m_i^\prime$ are possible. All values can be achieved because the welded crossing introduce no relations and you can always move strands into position using welded crossings, which do not add any extra information into the longitude in $\Gamma_i$.

 It follows from Theorem \ref{main1} that connect sum does not give a well defined group structure like in the two-component case for link maps.
\begin{cor}\label{tube}
There exists different choices of tubings for connect sum giving different resulting link homotopy classes in $\LM_{2,2,2}^4$.

\begin{proof}
Let $J$ be the braid described by Figure \ref{fig:PB} and let $X$ be the three-component $JK$ construction which keeps the $y$ component trivial and let $\overline{X}$ be the reversed mirror image of this $JK$ construction. Applying $\sigma^3$ to  the link maps $f_{XJ\overline{X}}$ and $f_{XJJ\overline{X}}$. From  Theorem \ref{main1} we have
\begin{align*}
\sigma^3\left(f_{XJ\overline{X}}\right)&=\left(z\left(s-1\right)+\overline{z}\left(s-1\right),0,x\left(1-u\right)+\overline{x}\left(1-\overline{u}\right)\right)\\
\sigma^3\left(f_{XJJ\overline{X}}\right)&=\left(z\left(s^2-1\right)+\overline{z}\left(\overline{s}^2-1\right),0,x\left(1-u^2\right)+\overline{x}\left(1-\overline{u}^2\right)\right).
\end{align*}
 Using Example \ref{examp}, we can see that these link maps are distinct from one another.
\end{proof}
\end{cor}
 It is worth noting that Theorem \ref{main1} and the proof of Proposition \ref{commute} show that the group of embedded annular link map up to link homotopy is not a normal subgroup of the group of annular link maps. 
 This is because if we take $J\in \ESL_{2,2,2}^4$ such that 
 \[
 \Theta(J)=\left((0,0,0), \eta\right)
 \]
 with $eta\neq\Id$. Let $X\in SL_{2,2,2}^4$ such that 
 \[
 \Theta(X)=\left((\theta_x(X), \theta_y(X), \theta_z(X)), \Id\right)
 \]
 with $(\theta_x(X), \theta_y(X), \theta_z(X))\neq(0,0,0)$. Then we can arrange that 
 \[
 \Theta(XJ\overline{X})=\left(a, \eta\right)\in L
 \]
 such that $a\neq 0$. However the $\ESL_{2,2,2}^4$ is contained in the following normal subgroup.
 \begin{prop}
Let $N$ be the set of annular link maps such that their link map closures are Brunnian. Then $N$ is a normal subgroup of the group of singular link concordances.
\begin{proof}
We first show $N$ is a subgroup and then prove that it is normal. 
Let $Y_1$ and $Y_2$ be annular link maps whose link map closures are Brunnian. We will show that $f_{Y_1Y_2}$ is a Brunnian link map. As $\Theta$ is a homomorphism
\[
\Theta\left(Y_1 Y_2\right)=\Big(\left(\theta_x(Y_1), \theta_y(Y_1), \theta_z(Y_1)\right)+\Phi(\eta(Y_1))\left(\left(\theta_x(Y_2), \theta_y(Y_2), \theta_z(Y_2)\right)\right), \eta(Y_1)\circ\eta(Y_2)\Big)\text{.}
\]
Taking the link map closure we know that for some choice of basing paths we have
\[
\widetilde{\sigma}^3\left(f_{Y_1Y_2}\right)=\left(\theta_x(Y_1), \theta_y(Y_1), \theta_z(Y_1)\right)+\Phi(\eta(Y_1))\left(\left(\theta_x(Y_2), \theta_y(Y_2), \theta_z(Y_2)\right)\right)\text{}
\]
where we have added using the basings provided by the annular link map.
Recall the map $p_i:\widetilde{K}\rightarrow \left(\mathbb{Z}\left[\mathbb{Z}\right]\right)^2$ where we project onto the factors which aren't $i$ and in each component set the $i$th meridian equal to $1$. We will do the case for $i=3$ as the other cases are similar
We have that
\[
p_3\left(\widetilde{\sigma}^3\left(f_{Y_1Y_2}\right)\right)=\left(\sigma_1(f_{Y_1^3})+\sigma_1(f_{Y_2^3}), \sigma_2(f_{Y_1^3})+\sigma_2(f_{Y_2^3})\right)
\]
Since the link map closures of each annular link map is trivial the terms in each sum is zero and by the injectivity of the Kirk invariant $Y_1Y_2$ is in $N$.

Suppose that $Y\in N$, we will show that $\overline{Y}\in N$. Note that
\[
\widetilde{\sigma}^3\left(f_{\overline{Y}}\right)=\Phi\left(\eta(Y)^{-1}\right)\left(-\theta_x(Y),-\theta_y(Y), -\theta_z(Y)\right)
\]
Applying the map $p_i$ for each $i$ and using the injectivity of the Kirk invariant we have $\overline{Y}\in N$.

As the trivial annular link map is clearly in $N$ we have shown that $N$ is a subgroup.

We now show that $N$ is normal. Let $X\in \SL_{2,2,2}^4$ and $Y\in N$. We will just focus on the intersection information of $XY\overline{X}$ since the automorphism data is not relevant when we map to $\LM_{2,2}^4$. The intersection information is given by
\[
\left(\theta_x(X),\theta_y(X), \theta_z(X)\right)+\Phi\left(\eta(X)\right)\left(\theta_x(Y),\theta_y(Y), \theta_z(Y)\right)+\Phi\left(\eta(XY)\right)\left(\theta_x(\overline{X}),\theta_y(\overline{X})), \theta_z(\overline{X})\right)\text{.}
\]
Since 
\[
\left(\theta_x(\overline{X}), \theta_y(\overline{X}), \theta_z(\overline{X})\right)=-\Phi(\eta(X)^{-1})\left(\left(\theta_x(X), \theta_y(X), \theta_z(X)\right)\right)
=-\Phi(\eta(\overline{X}))\left(\left(\theta_x(X), \theta_y(X), \theta_z(X)\right)\right)\text{,}
\]
the intersection information becomes
\[
\left(\theta_x(X),\theta_y(X), \theta_z(X)\right)+\Phi\left(\eta(X)\right)\left(\theta_x(Y),\theta_y(Y), \theta_z(Y)\right)-\Phi\left(\eta(XY\overline{X})\right)\left(\theta_x(X),\theta_y(X)), \theta_z(X)\right)
\]
We now write
\[
p_3\left(\widetilde{\sigma}^3\left(f_{XY\overline{X}}\right)\right)=\left(\sigma_1(X^3)-\sigma_1(X^3), \sigma_2(X^3)-\sigma_2(X^3)\right)=(0,0).
\]

By injectivity of the Kirk invariant removing the third component gives a trivial link map. A similar result can be shown for removing the first and second component and this proves the result.
\end{proof}
 \end{prop}
Hence, the normal closure of the set of three-component embedded link maps is contained within $N$.

\section{$n$-component link maps}\label{chap7}
We now define an $n$-component link map homotopy invariant. We will mostly sketch out the process and not prove things explicitly since the proofs will be analogous to the three-component case. We first need to prove the following proposition.
\begin{prop}
The fundamental group complement of  $k$ generically immersed spheres given by a link map is isomorphic to the free Milnor group on $k$ generators, where a meridian of each sphere is a generator.
\end{prop}
\begin{proof}
Let $f$ be an $k$-component link map and apply finger moves to arrange that the complement is a Milnor group. Using compactness we can apply a link homotopy to $f$ which moves it it below into the southern hemisphere of $S^4$. For each component of $f$, do a link  homotopy which takes a finger from each sphere and brings up to the northern hemisphere and the intersection with the equatorial $S^3$ is a single connected component, resulting in $n$-component unlink at the equator.

We then cut along the equator and this decomposes $S^4$ into two $D^4$ both with the same $n$-component unlink in their boundary $S^3$, where each component bounds an immersed disc in the four-ball and does not intersect any other component. On both discs we do finger moves to make the fundamental group of the complement a Milnor group. Using a result from \cite{DT} we know that the fundamental group of each of these complements is Milnor free group on $k$ generators. Using Seifert-Van Kampen this proves the result.
\end{proof} 
Let $f=f_1\sqcup \ldots f_n:S^2_1\coprod\ldots\coprod S^2_n\rightarrow S^4$ be a based link map. For each basing path we associate a meridian of the component , similar to the three-component case. Call these meridians $x_1,\ldots x_n$. Let
\[
\tau=\left(n, n-1,\ldots, 2,1\right)\in \mathcal{S}_n,
\]
recall $\mathcal{S}_n$ is the symmetric group on $n$ variables. Then define
\[
\Gamma_i:=\left<x_{\tau^{i-1}\left(1\right)}, x_{\tau^{i-1}\left(2\right)},\ldots, x_{\tau^{i-1}\left(n\right)}\mid x_{\tau^{i-1}\left(1\right)},\, r_1,\, r_2,\, r_3,\ldots,\, r_s\right>
\]
where $r_1,\ldots, r_s$ are the minimal relations for the free Milnor group on $n$ generators. Clearly each $\Gamma_i$ is the Milnor group on $n-1$ generators

\begin{defn}
Let $f=f_1\sqcup\ldots\sqcup f_n:S^2_1\coprod\ldots\coprod S^2_n\rightarrow S^4$ be a based link map then we define
\[
\sigma^n(f)=\left(M\left(\lambda\left(f_1, f_1\right)\right),\ldots, M\left(\lambda\left(f_i, f_i\right)\right),\ldots, M\left(\lambda\left(f_n f_n\right)\right)\right)\in \prod_{i=1}^n\mathbb{Z}\Gamma_i\text{.}
\] 
\end{defn}

\begin{prop}
Let $f$ and $g$ be link homotopic based link maps. Then
\[
\sigma^n(f)=\sigma^n(g)\text{.}
\]

\begin{proof}
The proof of this result is analogous to the proof of Proposition \ref{sigma}.
\end{proof}
\end{prop}

We now find the correct quotient to consider which removes the dependence on basing and gives a unbased link map invariant.

Choose another set of basing paths for $f$, similarly to the three component case, this pair of basing curves specifies elements of an $n$-tuple of $\left(g_1,g_2,\ldots,g_n\right)\in\prod_{i=1}^n\Gamma_i$. Let $\tilde{g}_i\in \Gamma_j$ where $i\neq j$ where $\tilde{g}_i$ is defined to be the image of $g_i$ under the sequence of maps
\[
\Gamma_i\twoheadrightarrow \Gamma_i/\left<\left<x_j\right>\right>\hookrightarrow \Gamma_j\text{,}
\]
where the left most arrow is the quotient map and the right most map is the natural inclusion map. The notation omits which generator we have omitted, however it will be clear from the context. We define the following map $\psi_{g_i}:\mathbb{Z}\Gamma_i\rightarrow \mathbb{Z}\Gamma_i$ where $\psi_{g_i}(r)=g_irg_i^{-1}$. Furthermore, let $\phi^i_{g_j}:\mathbb{Z}\Gamma_i\rightarrow \mathbb{Z}\Gamma_i$ with
\[
\phi^i_{g_j}(x_k)=\begin{cases}
\tilde{g}_j^{-1}x_k\tilde{g}_j & k=i\\
x_k &\text{ otherwise }
\end{cases}
\]
extended linearly. Define $w_i:\mathbb{Z}\Gamma_i\rightarrow \mathbb{Z}\Gamma_i$ where 
\[
w_i:=\psi_{g_{\tau^{i-1}\left(1\right)}}\circ\phi^i_{g_{\tau^{i-1}\left(2\right)}}\circ\phi^i_{g_{\tau^{i-1}\left(3\right)}}\circ\ldots\circ \phi^i_{g_{\tau^{i-1}\left(n\right)}}\text{.}
\]
Define an action by 
\[
\left(g_1,\ldots g_n\right)\cdot \left(r_1,\ldots,r_n\right)=\left(w_1\left(r_1\right), w_2\left(r_2\right),\ldots,w_n\left(r_n\right)\right)\text{.}
\]
This action corresponds to the changing basing paths similarly to the action defined in Lemma \ref{act} in the three-component case. We define our $n$-component invariant.
\begin{defn}
Let $f$ be an $n$-component link map. Choose a collection of basing paths for each sphere. Define
\[
\sigma_i\left(f\right):=M\left(\lambda\left(f_i, f_i\right)\right)\text{.}
\]
Define
\[
\widetilde{\sigma}^n\left(f\right)=\left(\sigma_1\left(f\right),\ldots\sigma_n\left(f\right)\right)\in\prod_{i=1}^n\mathbb{Z}\Gamma^i/\sim\text{.}
\]
\end{defn}
Let $1\leq i_1<\cdots <i_k\leq n$ let us define the map 
\[
\left(i_1,\ldots, i_k\right):\LM_{\underbrace{2,\ldots, 2}_{n}}^4\rightarrow \LM_{\underbrace{2,\ldots, 2}_{n-k}}^4
\]
where $\left(i_1,\ldots, i_k\right)(f)$ is the link map $f$ where we forget about each $i_j$ sphere. Define the $p_{\left(i_1,\ldots, i_k\right)}:\prod_{i=1}^n\mathbb{Z}\Gamma^i/\sim\rightarrow \prod_{i=1}^{n-k}\mathbb{Z}\Gamma^i/\sim$ to be the projection onto the factors which are not one of the $i_j$ and set all $x_{i_j}=1$ and relabel. 

\begin{prop}
The following diagram commutes
\[
\begin{tikzcd}
\LM_{2,\ldots, 2}^4\arrow[d, "\widetilde{\sigma}^n"]\arrow[r, "\underline{i}"]&\LM_{2,\ldots, 2}^4\arrow[d, "\widetilde{\sigma}^{n-k}"]\\
\prod_{i=1}^n\mathbb{Z}\Gamma^i/\sim\arrow[r]&\prod_{i=1}^{n-k}\mathbb{Z}\Gamma^i/\sim\text{,}
\end{tikzcd}
\]
where $\underline{i}:=\left(i_1,\ldots, i_k\right)$\text{.}
\end{prop}
The proof is omitted.

\end{document}